\documentclass[12pt]{article}
\usepackage{amsmath}
\usepackage{amsthm}
\usepackage{amssymb}
\usepackage{amsfonts}
\usepackage{todonotes}
\usepackage[round]{natbib}
\RequirePackage[colorlinks,citecolor=blue,urlcolor=blue]{hyperref}

\definecolor{ggplot_orange}{HTML}{F8766D}
\definecolor{ggplot_green}{HTML}{7CAE00}
\definecolor{ggplot_blue}{HTML}{00BFC4}

\newcommand{\E}{\mathop{\operatorname{E}\/}}

\newcommand{\Var}{\mathop{\operatorname{Var}\/}}
\newcommand{\Cov}{\mathop{\operatorname{Cov}\/}}

\newcommand \R{\mathbb R}
\newcommand \N{\mathbb N}
\newcommand \Z{\mathbb Z}

\newcommand{\ignore}&{}

\newcommand{\transp}{\textsc{t}}

\makeatletter
\newcommand*{\defeq}{\mathrel{\rlap{%
                     \raisebox{0.3ex}{$\m@th\cdot$}}%
                     \raisebox{-0.3ex}{$\m@th\cdot$}}%
                     =}
\makeatother

\makeatletter
\newcommand*{\eqdef}{=\mathrel{\rlap{%
                     \raisebox{0.3ex}{$\m@th\cdot$}}%
                     \raisebox{-0.3ex}{$\m@th\cdot$}}%
                     }
\makeatother


\newcommand{\pr}[1]{\left(#1\right)}

\usepackage{longtable}
\usepackage{graphicx}
\usepackage{csquotes} 
\usepackage{dsfont}
\usepackage{pdflscape}
\usepackage{morefloats}
\usepackage{tikz}
 \usetikzlibrary{patterns}

\theoremstyle{plain}
\newtheorem{theorem}{Theorem}[section]
\newtheorem{proposition}[theorem]{Proposition}
\newtheorem{lemma}[theorem]{Lemma}

\theoremstyle{definition}

\newtheorem{definition}[theorem]{Definition}

\newtheorem{example}[theorem]{Example}

\theoremstyle{remark}
\newtheorem{remark}[theorem]{Remark}

\usepackage{listings}
\begin{document}

\title{\bfseries  Depth Patterns and their Applications in Animal Tracking}

\author{%
    \textsc{Annika Betken}%
    \thanks{ University of Twente, Faculty of Electrical Engineering, Mathematics and Computer Science (EEMCS),  Drienerlolaan 5, 7522 NB Enschede, Netherlands, 
              \texttt{ a.betken@utwente.nl}
              . } \ and
                  \textsc{Alexander Schnurr}%
    \thanks{University of Siegen, Department of Mathematics,Walter-Flex-Str. 3,
               D-57072 Siegen, Germany,
              \texttt{Schnurr@mathematik.uni-siegen.de}
              . }
    }

\date{\today}

\maketitle

\begin{abstract}
We establish a definition of ordinal patterns for multivariate data sets based on the concept of Tukey's halfspace depth.
Given the definition of these \emph{depth patterns}, we are interested in the probabilities of observing specific patterns in time series. For this, we consider the relative frequency of depth patterns as natural estimators for their occurrence probabilities. Depending on the choice of reference distribution and the relation between reference and data distribution, we distinguish different settings that are considered separately. Within these settings we study statistical properties of depth pattern probabilities, establishing consistency and asymptotic normality  under the assumption of weakly dependent time series.
Since our concept only depends on ordinal depth information, the resulting values are robust under small perturbations and measurement errors.  We emphasize the applicability of our method by analyzing the depth patterns which are found in seal pubs' movement. We use our approach in order to choose an appropriate model out of a range of two-dimensional random walks, which are commonly used in mathematical biology. 
\end{abstract}


\noindent%
\vfill

\newpage

\section{Introduction}

Ordinal patterns encode the spatial order of temporally-ordered data points.
More precisely, by the ordinal pattern of order $p$ of data points $x_1, \ldots, x_p$ we refer to the permutation $(\pi_1,\ldots, \pi_p)$, where $\pi_j$ denotes the rank of $x_j$ within the values  $x_1, \ldots x_p$. For simplicity we assume that the values of the data points are all different. 

\begin{figure}[!ht]
\centering
\begin{tikzpicture}[x=1pt,y=1pt]
\path[use as bounding box,fill opacity=0.00] (0,0) rectangle (361.35, 72.27);
\begin{scope}
\path[draw=ggplot_orange,line width= 0.6pt] ( 10.68, 10.19) --
	( 34.93, 29.58) --
	( 59.17, 48.97);
\path[draw=ggplot_orange,line width= 0.4pt,line join=round,line cap=round,fill=ggplot_orange] ( 10.68, 10.19) circle (  2.50);
\path[draw=ggplot_orange,line width= 0.4pt,line join=round,line cap=round,fill=ggplot_orange] ( 34.93, 29.58) circle (  2.50);
\path[draw=ggplot_orange,line width= 0.4pt,line join=round,line cap=round,fill=ggplot_orange] ( 59.17, 48.97) circle (  2.50);
\end{scope}

\begin{scope}
\path[draw=ggplot_orange,line width= 0.6pt,line join=round] ( 69.53, 10.19) --
	( 93.77, 48.97) --
	(118.02, 29.58);
\path[draw=ggplot_orange,line width= 0.4pt,line join=round,line cap=round,fill=ggplot_orange] ( 69.53, 10.19) circle (  2.50);
\path[draw=ggplot_orange,line width= 0.4pt,line join=round,line cap=round,fill=ggplot_orange] ( 93.77, 48.97) circle (  2.50);
\path[draw=ggplot_orange,line width= 0.4pt,line join=round,line cap=round,fill=ggplot_orange] (118.02, 29.58) circle (  2.50);
\end{scope}

\begin{scope}
\path[draw=ggplot_orange,line width= 0.6pt,line join=round] (128.37, 48.97) --
	(152.62, 10.19) --
	(176.87, 29.58);
\path[draw=ggplot_orange,line width= 0.4pt,line join=round,line cap=round,fill=ggplot_orange] (128.37, 48.97) circle (  2.50);
\path[draw=ggplot_orange,line width= 0.4pt,line join=round,line cap=round,fill=ggplot_orange] (152.62, 10.19) circle (  2.50);
\path[draw=ggplot_orange,line width= 0.4pt,line join=round,line cap=round,fill=ggplot_orange] (176.87, 29.58) circle (  2.50);
\end{scope}

\begin{scope}
\path[draw=ggplot_orange,line width= 0.6pt,line join=round] (187.22, 48.97) --
	(211.47, 29.58) --
	(235.72, 10.19);
\path[draw=ggplot_orange,line width= 0.4pt,line join=round,line cap=round,fill=ggplot_orange] (187.22, 48.97) circle (  2.50);
\path[draw=ggplot_orange,line width= 0.4pt,line join=round,line cap=round,fill=ggplot_orange] (211.47, 29.58) circle (  2.50);
\path[draw=ggplot_orange,line width= 0.4pt,line join=round,line cap=round,fill=ggplot_orange] (235.72, 10.19) circle (  2.50);
\end{scope}

\begin{scope}
\path[draw=ggplot_orange,line width= 0.6pt,line join=round] (246.07, 29.58) --
	(270.32, 48.97) --
	(294.57, 10.19);
\path[draw=ggplot_orange,line width= 0.4pt,line join=round,line cap=round,fill=ggplot_orange] (246.07, 29.58) circle (  2.50);
\path[draw=ggplot_orange,line width= 0.4pt,line join=round,line cap=round,fill=ggplot_orange] (270.32, 48.97) circle (  2.50);
\path[draw=ggplot_orange,line width= 0.4pt,line join=round,line cap=round,fill=ggplot_orange] (294.57, 10.19) circle (  2.50);
\end{scope}

\begin{scope}
\path[draw=ggplot_orange,line width= 0.6pt,line join=round] (304.92, 29.58) --
	(329.17, 10.19) --
	(353.42, 48.97);
\path[draw=ggplot_orange,line width= 0.4pt,line join=round,line cap=round,fill=ggplot_orange] (304.92, 29.58) circle (  2.50);
\path[draw=ggplot_orange,line width= 0.4pt,line join=round,line cap=round,fill=ggplot_orange] (329.17, 10.19) circle (  2.50);
\path[draw=ggplot_orange,line width= 0.4pt,line join=round,line cap=round,fill=ggplot_orange] (353.42, 48.97) circle (  2.50);
\end{scope}

\definecolor{text}{gray}{0.10}
\begin{scope}
\node[text=text,anchor=base,inner sep=0pt, outer sep=0pt, scale=  0.80] at ( 34.93, -15.0) {(1, 2, 3)};
\end{scope}

\begin{scope}
\node[text=text,anchor=base,inner sep=0pt, outer sep=0pt, scale=  0.80] at ( 93.77, -15.0) {(1, 3, 2)};
\end{scope}

\begin{scope}
\node[text=text,anchor=base,inner sep=0pt, outer sep=0pt, scale=  0.80] at (152.62, -15.0) {(3, 1, 2)};
\end{scope}

\begin{scope}
\node[text=text,anchor=base,inner sep=0pt, outer sep=0pt, scale=  0.80] at (211.47, -15.0) {(3, 2, 1)};
\end{scope}

\begin{scope}
\node[text=text,anchor=base,inner sep=0pt, outer sep=0pt, scale=  0.80] at (270.32, -15.0) {(2, 3, 1)};
\end{scope}

\begin{scope}
\node[text=text,anchor=base,inner sep=0pt, outer sep=0pt, scale=  0.80] at (329.17, -15.0) {(2, 1, 3)};
\end{scope}
\end{tikzpicture}
\vspace{5mm}
\caption{The six univariate ordinal patterns of order 3 (not allowing for ties).}
\end{figure}

Since a definition of ordinal patterns presupposes a total ordering 
of the data, there is no straightforward extension of the notion of ordinal patterns from univariate to multivariate observations. 
Nevertheless,  applications  often  require an analysis of multivariate data sets:  physiological data sets such as ECG or EEG data are usually determined  from multiple electrodes. In portfolio optimization,  assets are supervised and modeled simultaneously. In ecology, the movement of animals on the ground is described by two coordinates. 

Different approaches have been suggested  to deal with multidimensional data via ordinal patterns; see Section \ref{sec:math} 
for an 
overview. Most of these treat the components of data vectors  separately. 
 Unlike other articles, which either only consider dependencies within each component (\cite{keller2003symbolic}, \cite{mohr2020new}) or only cross-dependencies between  components (\cite{schnurr:dehling:2017}, \cite{betken2021multdep}), we present a different approach  which incorporates both by taking into account the overall dynamics within the   multivariate reference system.

As motivation, imagine   a certain country having a high probability for earthquakes.  The location of each earthquake is determined according to the distribution $Q$. Either assume that the distribution $Q$ is known or that at least the time series of earthquake locations is  stationary. In the latter case, we get to know $Q$ better with each realization. 
Let $X_1, \dots, X_n$ denote the geographic coordinates of an animal at time points $j=1, \ldots, n$, and assume that the animal's movement is independent of $Q$.
In a first step, questions of interest in this context could be the following:
How close is the animal to the (potential) center of an earthquake region? Is it more likely to move  towards or from the center? In a second step one might consider models that allow the movement of animals to
depend on the 
distribution of seismic activity. 

In any case, these questions can be answered by defining ordinal patterns of  geographic coordinates
on the basis of `how deep' the coordinates lie in the earthquake region. 
As a result, a lack of canonical ordering of  $\mathbb{R}^2$ 
can, in this case,  be overcome by the concept of statistical depth.
The basic idea of statistical depth
 is to measure how deep a specific element in a multidimensional space lies in a given, multivariate  reference distribution.
 This leads naturally to a 
center-outward ordering of sample points in multivariate data. 
Starting with Tukey's proposal of  {\em halfspace depth} in 1975,  a number of different depth functions have been proposed.
With respect to the questions raised against the background of the considered motivational example,
however, Tukey's original concept of statistical depth is the most suitable choice. 
Therefore, 
our aim is to estimate the probability distribution of ordinal patterns defined with respect to Tukey's halfspace depth. This knowledge can be used, for example, in the context of model selection.  


The paper is structured as follows: In the subsequent section we fix  notations and provide  mathematical definitions. Furthermore, we give a short survey on the existing literature. Section 3 establishes limit theorems in the setting where the reference function is known. The case of an unknown reference function is considered in Section 4. 
In Section 5 we show the applicability of our methods in a real world data example. A short discussion in Section 6 rounds out the paper. 

\section{Mathematical and historical background} \label{sec:math}

In this section, starting from the definition of ordinal patterns for univariate observations, we establish the concept of depth patterns  as its multivariate analogue.
For this, we  base an ordering of multivariate observations on the concept of statistical depth.
Apposite to the statistical applications motivating our results we choose Tukey's halfspace depth
as basis for our  conception of depth patterns in this article.
The definition of halfspace depth requires a reference distribution with respect to which  depths of observations and accordingly their depth patterns are computed. 
Depending on the choice of reference distribution and the relation between reference and data distribution, 
we distinguish four different cases that are  considered separately in the following sections.

For one-dimensional observations ordinal patterns are defined as follows:
\begin{definition} \label{def:onedimpat}
    For $p\in\mathbb{N}$ let $S_p$ denote the set of permutations of $\{1, \ldots, p\}$, which we write as $p$-tuples containing each of the numbers $1, \ldots, p$ exactly once.
By the \emph{ordinal pattern of order $p$} of observations $x_1, \ldots, x_p$  we refer to the permutation
\begin{align*}
\Pi(x_1,..., x_p)=(r_1,..., r_p)\in S_p
\end{align*}
which satisfies
\[
r_j \leq r_k  \hspace{10mm} \Longleftrightarrow \hspace{10mm}   x_j \leq x_k
\]
for every $j,k\in \{1,2,...,p\}$ with $r_j<r_k$ if $x_j=x_k$ for $j < k$.
\end{definition}



The ordinal pattern consists of the ranks of the values of the original vector. Originally, ordinal patterns have been introduced  
to measure the complexity 
of data sets (and their underlying models) by means of  the so-called permutation entropy (the Shannon entropy of a random variable taking values in $S_p$ where each $\pi \in S_p$ occurs with probability
$p(\pi):=P(\Pi(X_1,\ldots,X_{p})=\pi)$); see  \cite{bandt:pompe:2002}.
Since this seminal paper, ordinal patterns have proved useful for the analysis of different types of data sets such as EEG data (\cite{keller:maksymenko:stolz:2015}), speech signals (\cite{bandt:2005}), and chaotic maps which relate to the theory of dynamical systems (\cite{bandt:pompe:2002}).
Further applications include the approximation of the Kolmogorov-Sinai entropy; see \cite{sinn:2012}. 
More recently, ordinal patterns have been used to detect and to model the dependence between time series; see \cite{schnurr:2014}. Limit theorems for the parameters under consideration have been proved for short-range dependent time series in \cite{schnurr:dehling:2017}.

Due to the lack of  a total ordering of points in $\mathbb{R}^d$, 
an order relative to some reference object is required.
Implicitly such a relative order is provided by the concept of statistical depth, which quantifies 
the deepness of data points
relative to a given, multivariate  reference distribution.
Starting with Tukey's proposal of  {\em halfspace depth} (also called {\em location depth} or {\em Tukey
depth}) in 1975,  a number of different depth functions have been proposed (see
\cite{donoho1992breakdown}).
\begin{definition}[Halfspace Depth] \label{def:tukey}
Let $Q$ be a probability distribution on $\R^d$ and let $\mathcal{S}^{d-1}:=\left\{x\in \mathbb{R}^d:\|x\|=1\right\}$ denote the unit sphere in $\mathbb{R}^d$. The halfspace depth $D_Q(x)$ of a point $x\in \mathbb{R}^d$ with respect to  $Q$  is defined as
\begin{align*}
D_Q(x):=\min\limits_{\phi\in \mathcal{S}^{d-1}}Q\left(\left\{z\in \mathbb{R}^d|(z-x)^T\phi\geq 0\right\}\right).
\end{align*}
\end{definition}

Let us mention that  \cite{liu1990notion} and \cite{zuo2000general}   axiomatically approached the definition of data depth by  establishing four properties
that  any reasonable statistical depth function should have.
Halfspace depth
satisfies these axioms for  all absolutely continuous reference distributions on $\R^d$. Many other appealing properties of halfspace depth are well-known
and well-documented; see, e.g., \cite{donoho1992breakdown}, \cite{mosler2002multivariate}, \cite{koshevoy2002tukey},
\cite{ghosh2005maximum}, \cite{cuesta2008random}, \cite{hassairi2008tukey}.

Generally speaking, the choice of depth function depends on the particular statistical application one is interested in. 
Halfspace depth is a natural choice for  elliptical reference distributions, which correspond to the type of distributions motivating our results.
\begin{definition}
    A $d$-dimensonal random vector $X$ is said to have an elliptical distribution with $\mu\in \mathbb{R}^d$, positive definite 
    symmetric $d\times d$ matrix $\Sigma$, and probability density function $h$, if 
      its density is of the form
\begin{align*}
    f(x)=c|\Sigma^{-\frac{1}{2}}|h\left((x-\mu)^{\transp}\Sigma^{-1}(x-\mu)\right);
\end{align*}
where $c$ is a positive constant; see \cite{zuo2000structural}.
\end{definition}


Relying on the depth function $D_Q$ with reference distribution $Q$, we define the depth patterns of a point in $\mathbb{R}^d$:
\begin{definition}
For a finite sequence $(x_j)_{j=1,...,p}$ in $\R^d$, we define the 
\emph{depth pattern} $\Pi_Q$
by the vector $\tau =(\tau(1), \ldots, \tau(p))\in \N^p$ satisfying
\begin{align*}
\Pi_Q(x_1, \ldots, x_p)=(\tau(1), \ldots, \tau(p)), 
\end{align*}
where \begin{align*}
\tau(i)\defeq \# \left\{j\in \{1, \ldots,  p \}\left|\right. D_Q(x_j)\geq D_Q(x_i)\right\}.
\end{align*}
\end{definition}
In the above definition,
the integer $\tau(i)$ describes how deep the $i$-th entry of the vector $(x_1, \ldots, x_p)$ lies with respect to the reference measure $Q$. The deeper $x_i$ in $Q$, the smaller $\tau(i)$. Lemma \ref{lem:unique} in Section \ref{sec:first_cases} of this article establishes assumptions guaranteeing that the depths of two points $x_i$ and $x_j$ are (almost surely) pairwise different. In this case, $\Pi_Q(x_1, \ldots, x_p)$ corresponds to a permutation of the indices of $x_1, \ldots, x_p$. 

In Section \ref{sec:second_cases} of the present article we approximate the measure $Q$ by a sequence of discrete measures $Q^{(n)}$. In this setting, ties appear naturally in $\Pi_{Q^{(n)}}(x_1, \ldots, x_p)$. However, we will see  that with increasing sample size these vanish if the limiting distribution $Q$ of $Q^{(n)}$ does not allow for ties to occur.

Given ties, $\Pi_{Q^{(n)}}(x_1, \ldots, x_p)$ takes values in a space isomorphic to the set of Cayley permutations:
Accordingly, each entry of $(\tau(1),...,\tau(p))$ takes values in $\{1,...,p\}$, but with the additional restriction that, 
if $x_{i_1}, \ldots, x_{i_k}$
have the same depth with respect to $Q^{(n)}$, $\tau(i_1)=\ldots =\tau(i_k)$.
Consequently, it holds that, 
if $j,j+1,...j+k-1$ do not appear as an entry, but $j+k$ does, then $j+k$ appears $k+1$ times. 
For example, if $p=4$, the values $(4,4,4,4)$ and $(2,4,4,2)$ are possible, while (2,2,2,2) is not. For this subspace of $\mathbb{N}^p$ we write $T_p$ in analogy to $S_p$ (the permutations of length $p$). 

Given
the definition of depth patterns, we are interested in the probability 
of observing a specific pattern $\pi$ in 
a time series $(X_j)_{j \in \N}$.
For this, we consider the relative frequency
\begin{align}\label{def:estimator}
  \hat{p}_{n, Q}(\pi):=\frac{1}{n-p+1}\sum\limits_{i=0}^{n-p+1}1_{\left\{\Pi_Q(X_{i},X_{i+1}, \ldots, X_{i+p-1})=\pi\right\}}
\end{align}
of the depth pattern $\pi\in T_p$
as  a natural estimator for
\begin{align*}
p_Q(\pi):=P\pr{\Pi_Q(X_1,\ldots, X_{p})=\pi}.
\end{align*}
This estimator is the main object of our studies in Sections \ref{sec:first_cases}
and \ref{sec:second_cases}.

Given stationary time series   $(X_j)_{j \in \N}$ with values in 
$\mathbb{R}^d$, $d>1$, and marginal distribution $P_{X_1}$, 
and a distribution $Q$ with respect to which  statistical depth is defined,
its analysis depends on whether the reference distribution $Q$ is known or approximated by its empirical analogue $Q^{(n)}$ and whether $Q$ and $P_{X_1}$ coincide.
Accordingly, we distinguish the following cases:
\begin{enumerate}
    \item[(A)] $Q=P_{X_1}$ and $Q$ is known. 
    \item[(B)] $Q$ is known, but the relationship to $P_{X_1}$ is not specified. In particular, if $Q=P_Z$ the random variable $Z$ might be independent of the time series $(X_j)_{j \in \N}$. 
        \item[(C)] $Q=P_{X_1}$ is unknown and observed through the time series  $X_1, \ldots, X_n$.
    \item[(D)] $Q$ is unknown. We observe $Q$ through observations $Y_1, \ldots, Y_m$ and determine the depth of $X_1, \ldots, X_n$ with respect to the empirical distribution of $Y_1, \ldots, Y_m$.  In particular, the  two distributions may, but do not have to be, independent.
\end{enumerate}

We close this section by giving a short overview on how other authors have treated multivariate time series (or data sets) using ordinal patterns. 

Given multivariate data sets $(x_j)_{j=1,...,n}$ with $x_j=(x_{j,1}, ..., x_{j,d})$
the following approaches to defining ordinal patterns have been discussed in the literature so far:
\begin{enumerate}
    \item \cite{keller2003symbolic} determine the univariate ordinal patterns of $x_{1, i}, \ldots, x_{n, i}$ for each  $i$ and, subsequently, average  over the dimensions $i=1, \ldots, d$. The interplay between the different dimensions is  neglected in this approach. 
\item   \cite{mohr2020new} determine the univariate ordinal patterns of $x_{1, i}, \ldots, x_{n, i}$ for each  $i$ and, subsequently,  store all $d$ patterns at a fixed time point $i$ in one vector. The multivariate pattern is, hence, a vector of univariate patterns. The number of patterns is $p!\cdot d$.
\item \cite{he2016multivariate} determine the  univariate ordinal patterns (of length $d$) of $x_{t, 1}, \ldots, x_{t, d}$, and, subsequently, average over all $n$ time points. For patterns with ties, this is the spatial approach described in \cite{schnurr2022generalized}.
\item  \cite{rayan2019multidimensional} (amongst others) project the multivariate data into a one-dimensional space first, and, subsequently, determine the ordinal patterns of the projected values in the clasical way.   
\end{enumerate}
Against the background of the approaches described above, the consideration of depth patterns is closest 
to techniques that make use of dimension reduction, i.e., the fourth approach.

\section{Depth patterns based on a known reference distribution} \label{sec:first_cases}

In this section we consider the cases in which the reference distribution $Q$ is known, i.e., cases (A) and (B) mentioned in Section \ref{sec:math}. In these cases, it is possible to reduce proofs to classical limit theorems.
For the depth pattern estimator $\hat{p}_{n, Q}$ defined in \eqref{def:estimator} we show consistency for time series stemming from stationary, ergodic processes 
and asymptotic normality for time series corresponding to  stationary, 1-approximating functionals of
 absolutely regular processes.
Additionally, we provide easy to check, sufficient criteria under which asymptotic normality of the estimator is mathematically guaranteed.

Consistency is implied by Proposition 1 in \cite{betken2021ordinal}
which is a consequence of the Birkhoff-Khinchin ergodic theorem (see Theorem 1.2.1 in \cite{cornfeld2012ergodic}):
\begin{proposition}\label{prop_1:consistent}
Let $Q$ be a probability distribution on $\R^d$.
Suppose that $(X_j)_{j \in \N}$ is a stationary ergodic time series with values in $\R^d$. Then, $\hat{p}_{n, Q}(\pi)$ is a consistent estimator of $p_Q(\pi):=P\pr{\Pi_Q(X_1,\ldots,X_{p})=\pi}$. More precisely,
\begin{align} \label{firstcon}
 \lim_{n\rightarrow \infty }\hat{p}_{n, Q}(\pi)=p_Q(\pi)
\end{align}
almost surely.
\end{proposition}


In all that follows, the interplay between the measures $Q$ and $P_{X_1}$ is  important. 
Generally speaking,  for our analysis we need to avoid ties in the depths of the considered data, that is, we require different observations to have different depths.
Due to its significance, we define the corresponding property as \emph{separation by depth}:
\begin{definition}
We say that $Q$ \emph{separates  $X=(X_j)_{j \in \N}$ by depth}, if the probability that $X_j$ and $X_k$ have the same depth for $j\neq k$ is zero. 
\end{definition}

In order to derive the asymptotic distribution of the estimator $\hat{p}_n(\pi)$, we have to make  some assumptions on the dependence structure of the data-generating time series.
We assume that $(X_j)_{j \in \N}$ is a functional of an absolutely regular time series  $(Z_n)_{n \in \Z}$. 
For this, we recall the following concepts: 
\begin{definition}
Let $\left(\Omega, \mathcal{F}, P\right)$  be a probability space. Given two sub-$\sigma$-fields
$\mathcal{A}, \mathcal{B}\subset \mathcal{F}$, we define
\begin{align*}
\beta(\mathcal{A}, \mathcal{B})\defeq \sup\sum\limits_{i, j}\left|P\left(A_i\cap B_j\right)-P\left(A_i\right)P\left(B_j\right)\right|,
\end{align*}
where the supremum is  taken over all partitions $A_1, \ldots, A_{I}\in \mathcal{A}$ of $\Omega$, and over all partitions $B_1, \ldots, B_J\in \mathcal{B}$ of $\Omega$.
\end{definition}

\begin{definition}
The time series $(Z_n)_{n \in \N}$ is called {\em absolutely regular} with coefficients $\beta_m$, $m\geq 1$, if
\begin{align*}
\beta_m\defeq \sup_{n\in \mathbb{Z}}\beta\left(\mathcal{F}_{-\infty}^{n}, \mathcal{F}_{n+m+1}^{\infty}\right)\longrightarrow  0,
\end{align*}
as $m\rightarrow\infty$. 
Here, $\mathcal{F}_k^l$ denotes the $\sigma$-field generated by the random variables $Z_k,\ldots ,Z_l$.
\end{definition}
Let $(X_j)_{j \in \N}$, be an $\mathbb{R}^d$-valued stationary time series, and let $(Z_n)_{n \in \Z}$, be a stationary time series with values in some measurable space $S$. We say that $(X_j)_{j \in \N}$ is a functional of the time series $(Z_n)_{n\in \Z}$ if there exists a measurable function $f:S^{\mathbb{Z}}\longrightarrow \mathbb{R}^d$ such that, for all $j \in \N$,
\begin{align*}
X_j=f\left((Z_{j+n})_{n\in \mathbb{Z}}\right).
\end{align*}

\begin{definition}
We call $(X_j)_{j\in \N}$ a {\em 1-approximating functional} with constants $a_m$, $m\geq 1$, if for any $m\geq 1$ , there exists a function 
$f_m:S^{2m+1}\longrightarrow\mathbb{R}^d$ such that for every $i\in \mathbb{N}$
\begin{align*}
\E \|X_i-f_m(Z_{i-m}, \ldots, Z_{i+m})\|\leq a_m.
\end{align*}
\end{definition}

This class provides a good compromise between richness and tractability. In particular it contains many useful subclasses and examples like the MA($\infty$) process, Baker's map and continuous fractions (cf. \cite{schnurr:dehling:2017} Section 2.2). 

Given the above definitions, we state the main result of this section for  corresponding time series:
\begin{theorem}\label{thm:limit_theorem_1}
Let $(X_j)_{j \in \N}$ be a stationary 1-approximating functional of the absolutely regular time series $(Z_n)_{n \in \Z}$. Let  $\beta_k$, $k\geq 1$,  denote the mixing coefficients of the time series $(Z_n)_{n \in \Z}$, and let $a_k$, $k\geq 1$, denote the 1-approximating constants. Assume that
\begin{align*}
\sum\limits_{m=1}^{\infty}\sqrt{a_m}< \infty \ \text{ and }  \ \sum\limits_{k=1}^{\infty}\beta_k<\infty.
\end{align*}
Furthermore, assume that $Q$ separates $(X_j)_{j \in \N}$ by depth and that the distribution functions of $D_Q(X_{j})-D_Q(X_1)$ are  Lipschitz continuous for any $j\in \{2, \ldots, p\}$.
Additionally, assume that Tukey's depth with respect to $Q$ is Lipschitz continuous. 
 Then, as $n\rightarrow \infty$, 
\begin{align*}
\sqrt{n}\left(\hat{p}_n(\pi)-p(\pi)\right)\overset{\mathcal{D}}{\longrightarrow} N(0, \sigma^2), 
\end{align*}
where 
\begin{align*}
\sigma^2\defeq \Var\left(1_{\{\Pi_Q(X_{1}, \ldots, X_{p})=\pi\}}\right)+2\sum\limits_{m=2}^{\infty}\Cov\left(1_{\{\Pi_Q(X_{1}, \ldots, X_{p})=\pi\}}, 1_{\{\Pi_Q(X_{m}, \ldots, X_{m+p-1})=\pi\}}\right).
\end{align*}
\end{theorem}

\begin{remark}
Let us emphasize, that we do not explicitly need properties of Tukey's depth in order to prove the theorem. Using other depth concepts is also possible, as long as Lipschitz continuities and the separation by depth property are satisfied. 
\end{remark}

As discussed earlier, the assumption imposed on $(X_j)_{j \in \N}$ in Theorem \ref{thm:limit_theorem_1}
can be considered non-restrictive.
It is however, less obvious to explicitly state examples for which the conditions relating  $(X_j)_{j \in \N}$  and $Q$ hold. In the following, we provide such an example:

\begin{example}
For all practical purposes we are considering here, one is typically interested in time series with values in a bounded domain (earthquakes, animal movement, rainfall). For simplicity, we consider $B_1(0)$, the open disc around zero with radius 1. Moreover, we consider any time series $(X_j)_{j \in \N}$ as in Theorem \ref{thm:limit_theorem_1} having values in $B_1(0)$. In this situation, we are aiming at finding  a distribution $Q$ satisfying both Lipschitz conditions in Theorem \ref{thm:limit_theorem_1}. 
For this, pick what might be the simplest choice for a depth function, that is, 
\begin{align} \label{wishdepth}
D(v):=\frac{1}{2}(1-||v||)\cdot 1_{B_1(0)}.
\end{align}
The depth here only depends on the distance to zero. The maximum is attained at the origin. The first (non-trivial) question is, whether a distribution $Q$ exists, which yields this depth function. It follows from \eqref{wishdepth} that this distribution has to be rotation invariant and that it has to have uniform marginal distributions. In \cite{perlman2011squaring}, the existence of such a distribution $Q$ is shown. It is unique and has the probability density function
\[
f(v_1,v_2)=\frac{1}{2\pi \sqrt{1-v_1^2-v_2^2}} \text{ on } v_1^2+v_2^2<1.
\]
Assume that the marginals of the time series $(X_j)_{j \in \N}$ admit continuous probability density functions (concentrated on $B_1(0)$ ). Then, $Q$ separates $(X_j)_{j \in \N}$ by depth and $D_Q$ is Lipschitz continuous. What remains to be shown is that the distribution functions of $D_Q(X_{j})-D_Q(X_1)$ are  Lipschitz continuous. To this end, we show that the respective probability density functions are bounded: 

Fix $j\in\N$. By definition, we have $D_Q(X_j)-D_Q(X_1)=(1/2)(||X_1||-||X_j||)$.
In order to show that the density of this random variable is bounded, we transform the density of $(X_1,X_j)'$ in the following way: 
\[
\Phi:\left(\begin{array}{c} x_1 \\ x_2 \\ y_1 \\ y_2 \end{array}\right) \mapsto
\left(\begin{array}{c} (1/2)(\sqrt{x_1^2+x_2^2}-\sqrt{y_1^2+y_2^2}) \\ x_2 \\ y_1 \\ y_2 \end{array}\right).
\]
On $A:=B_1(0)\cap (0,1)\times (0,1)$ this is a diffeomorphism. The same holds true on the other quadrants. We consider here only the first quadrant in order to keep the notation simple. The inverse function is
\[
\Phi^{-1}:\left(\begin{array}{c} z_1 \\ z_2 \\ z_3 \\ z_4 \end{array}\right) \mapsto
\left(\begin{array}{c} \sqrt{(2z_1+\sqrt{z_3^2+z_4^2})^2-z_2^2} \\ z_2 \\ z_3 \\ z_4 \end{array}\right)
\]
By the transformation theorem the probability density function of $D_Q(X_j)-D_Q(X_1)$ on $\Phi(A)$ is
\[
f(z_1):=\int_0^1 \int_0^{1-z_4^2} \int_0^{2z_1+\sqrt{z_3^2+z_4^2}} g(\Phi^{-1}(z)) |\text{det} J_{\Phi^{-1}}| dz_2 dz_3 dz_4
\]
where $g$ is the bounded Lebesgue density of $(X_1, X_j)'$. Since $g(\Phi^{-1}(z))$ is bounded, it can be taken out of the integrals. The determinant of the Jabobian is
\[
\frac{2\cdot(2z_1+\sqrt{z_3^2+z_4^2})}{\sqrt{(2z_1+\sqrt{z_3^2+z_4^2})^2-z_2^2}}. 
\]
Integrating this function in $z_2$ we obtain as the primitive function 
\[
2c \arcsin\left(  \frac{z_2}{c}
\right)
\]
with $c=2z_1+\sqrt{z_3^2+z_4^2}$ being bounded on the domain. Since arcsin is also bounded, we obtain the probability density function $f$ as a double integral of a bounded function on a bounded domain. Hence, $f$ is bounded and the assumptions of the theorem are satisfied.

\end{example}

\begin{proof}[Proof of Theorem \ref{thm:limit_theorem_1}]
We apply Theorem 18.6.3 of \cite{ibragimov:linnik:1971}
to the partial sums of the random variables $(Y_j)_{j \in \N}$
defined by
\begin{align*}
Y_j\defeq 1_{\{\Pi_Q(X_{j+1}, \ldots, X_{j+p})=\pi\}}.
\end{align*}
According to the following lemma 
$(Y_j)_{j \in \N}$  is a 1-approximating functional of the time series $(Z_n)_{n \in \Z}$ with approximating constants $\sqrt{a_k}$, $k\geq 1$. Thus, the conditions of Theorem 18.6.3 of \cite{ibragimov:linnik:1971} are satisfied.
\end{proof}

\begin{lemma}\label{lem:1-approx}
Let  $(X_j)_{j \in \N}$ be a 1-approximating functional of the time series $(Z_n)_{n \in \Z}$ with approximating coefficients $a_m$, $m\in \mathbb{N}$. Assume that $Q$ separates  $(X_j)_{j \in \N}$ by depth and that the distribution function of $D_Q(X_j)-D_Q(X_1)$ is Lipschitz continuous for any $j\in \N$.
Additionally, assume that Tukey's depth with respect to $Q$ is Lipschitz continuous.
Then, for any depth pattern $\pi$,
\begin{align*}
W_i\defeq 1_{\{\Pi_Q(X_{i+1}, \ldots, X_{i+p})=\pi\}}
\end{align*}
is a 1-approximating functional of the time series $(Z_j)_{j \in \N}$ with approximating coefficients $\sqrt{a_m}$, $m\in \mathbb{N}$.
\end{lemma}
\begin{remark}
As for Theorem \ref{thm:limit_theorem_1} let us emphasize, that we do not explicitly need properties of Tukey's depth in order to prove the lemma. Using other depth concepts is also possible, as long as Lipschitz continuity and the separation by depth property are satisfied. 
\end{remark}

\begin{proof}
Define $X_i^{(m)}=f_m(Z_{i-m}, \ldots, Z_{i+m})$ and
$W_i^{(m)}\defeq 1_{\{\Pi_Q(X_{i+1}^{(m)}, \ldots, X_{i+p}^{(m)})=\pi\}}$.
Observe that for all $\epsilon > 0$ and all integers $i\geq 1$ the following implication holds: 
If $\left|D_Q(X_{i+j})-D_Q(X_{i+k})\right|\geq 2\epsilon$
for all $0\leq j< k \leq p$
and $\left|D_Q(X_{i+j})-D_Q(X^{(m)}_{i+j})\right|\leq \epsilon$
for all $0\leq j \leq p$, then  $\Pi_Q\left(X_1, \ldots, X_p\right)= \Pi_Q\left(X^{(m)}_1, \ldots, X^{(m)}_p\right)$.
As a result,  $\Pi_Q\left(X_1, \ldots, X_p\right)\neq \Pi_Q\left(X^{(m)}_1, \ldots, X^{(m)}_p\right)$ implies that either 
the difference of the depths of $X_i$ and $X_j$ is smaller than 
$2\varepsilon$ for some $i, j\in \{1, \ldots, p\}$, $i\neq j$, or the difference in depths of $X_i$ and $X^{(m)}_i$ is bigger than $\varepsilon$ for some
$i\in \{1, \ldots, p\}$.

Then, for all $\varepsilon>0$
\begin{align*}
&\E \left(\left|W_i-W_i^{(m)}\right|\right)\\
&\leq \E \left(1_{\{\Pi_Q(X_{i+1}^{(m)}, \ldots, X_{i+p}^{(m)})\neq \Pi_Q(X_{i+1}, \ldots, X_{i+p})\}}\right)\\
&\leq \sum\limits_{j\neq k}P \left(|D_Q(X_{i+j})-D_Q(X_{i+k})|<2\varepsilon\right)+ \sum\limits_{j=1}^pP \left(|D_Q(X_{i+j}^{(m)})-D_Q(X_{i+j})|>\varepsilon\right)\\
&\leq \sum\limits_{j\neq k}P \left(|D_Q(X_{i+j})-D_Q(X_{i+k})|<2\varepsilon\right)+ \sum\limits_{j=1}^pP \left(L\|X_{i+j}^{(m)}-X_{i+j}\|>\varepsilon\right)\\
&\leq p(p-1)2C\varepsilon+\frac{L}{\varepsilon} \sum\limits_{j=1}^pE \left(\|X_{i+j}^{(m)}-X_{i+j}\|\right)\\
&\leq p(p-1)2C\varepsilon+\frac{Lp}{\varepsilon}a_m
\end{align*}
with $C$ denoting the Lipschitz constant of the distribution function of $D_Q(X_j)-D_Q(X_1)$.
Choosing $\varepsilon=\sqrt{a_m}$, the assertion follows. 
\end{proof}

The following result gives sufficient conditions for $Q$ to separate a time series by depth.

\begin{lemma}\label{lem:unique}
Assume  that the reference measure $Q$ corresponds to an elliptical distribution and that all two-dimensional distributions of the time series $(X_j)_{j \in \N}$
have a density with respect to the Lebesgue measure on $\mathbb{R}^d$
 and that the support of the marginal distribution $F$ of $X_1$ is contained in the support of $Q$.
Then, 
the probability that two datapoints $X_i$ and $X_j$ have the same depth is zero. 
\end{lemma}

\begin{proof}
According to Theorem 3.4 in \cite{zuo2000structural} 
the depth contours of halfspace depth are surfaces of ellipsoids and, therefore, nullsets with respect to the Lebesgue measure on $\mathbb{R}^d$. 
     Since all  bivariate marginal distributions of $(X_j)_{j\in \mathbb{N}}$ have a  density with respect to the Lebesgue measure,   the probability of two samples $X_i$ and $X_j$ having the same depth is zero.
\end{proof}



\section{Depth patterns based on an unknown reference distribution}  \label{sec:second_cases}

In this section, we consider the depth patterns of time series  $X=(X_j)_{j\in \mathbb{N}}$ with respect to an unknown reference distribution $Q$, i.e., cases (C) and (D) distinguished in Section \ref{sec:math}. 
However, $Q$ can be approximated  through observations generated by a stationary ergodic time series $Y=(Y_j)_{j\in \mathbb{N}}$  with marginal distribution $Q$.
We assume that the two time series $X$ and $Y$ are defined on the same probability space $\left(\Omega, \mathcal{F}, P\right)$.
$X$ and $Y$ may be independent, dependent or even the same (situation (C)). 


Given the above described setting, we study 
\[
  \hat{q}_{n,m, Q}(\pi):=\frac{1}{n-p+1}\sum\limits_{i=1}^{n-p+1}1_{\left\{\Pi_{Q^{(m)}}(X_{i}, \ldots, X_{i+p-1})=\pi\right\}},
\]
where 
\begin{align*}
Q^{(m)}(\omega)=\frac{1}{m}\sum\limits_{j=1}^m\delta_{Y_j(\omega)}.
\end{align*}
Our goal is to establish a limit theorem such as Proposition \ref{prop_1:consistent},  this time replacing $Q$ by $Q^{(m)}$ in the approximating sequence.
To start with, imagine
 that the time series data given by $X$ is fixed, while the number of observations $n$ from $Y$ increases, that is, we get to know  the reference distribution $Q$ better as $n$ tends to infinity. It could happen that then the depth pattern of, say, the data points $x_1, x_2, x_3 \in \mathbb{R}^2$ changes as more data points of $Y$ emerge; see Figure \ref{fig:empirical_depth}.

\begin{figure}
\begin{center}
\scalebox{0.5}{
\begin{tikzpicture}[label distance=-1mm]
    \begin{scope}[thick,font=\large]
\fill[ggplot_green!60]      (0,4) -| (4,0) -- cycle;
\fill[ggplot_orange!60]    (-0.5,4)   -| (-3,-4) -- cycle;
\fill[ggplot_blue!60]    (-3,-4)   -| (4,-1) -- cycle;
\draw[ggplot_green] (0,4) -- (4,0) node [above left] {};
\draw[ggplot_orange] (-0.5,4) -- (-3,-4) node [above left] {};
\draw[ggplot_blue] (-3,-4) -- (4,-1) node [above left] {};

    \draw [->] (-3,0) -- (4.2,0) node [above left] {};
    \draw [->] (0,-4) -- (0,4.2) node [below right] {};

\draw (3,3) circle[radius=4pt];
\draw (-1,1) circle[radius=4pt];
\draw (1,-2) circle[radius=4pt];
\node[circle,inner sep=4pt,fill=black,label= above left:{$x_1$}] at (2,2) {};
\node[circle,inner sep=4pt,fill=black,label= above left:{$x_2$}] at (-2.1,-1) {};
\node[circle,inner sep=4pt,fill=black,label= above left:{$x_3$}] at (-1,-3.15) {};
\node (a) at (0.5,-4.8) {$D(x_1)$ = 1/3, $D(x_2)$ = 0, $D(x_3)$ = 0};
    \end{scope}
\end{tikzpicture}


\begin{tikzpicture}[label distance=-1mm]
    \begin{scope}[thick,font=\large]

\fill[ggplot_green!60]      (0,4) -| (4,0) -- cycle;
\fill[ggplot_orange!60]    (-0.5,4)   -| (-3,-4) -- cycle;
\fill[ggplot_blue!60]    (-3,-4)   -| (4,-1) -- cycle;
\draw[ggplot_green] (0,4) -- (4,0) node [above left] {};
\draw[ggplot_orange] (-0.5,4) -- (-3,-4) node [above left] {};
\draw[ggplot_blue] (-3,-4) -- (4,-1) node [above left] {};

    \draw [->] (-3,0) -- (4.2,0) node [above left] {};
    \draw [->] (0,-4) -- (0,4.2) node [below right] {};

\draw (3,3) circle[radius=4pt];
\draw (-1,1) circle[radius=4pt];
\draw (1,-2) circle[radius=4pt];
\draw (-1,-2) circle[radius=4pt]; 
\draw (-2,-3) circle[radius=4pt]; 
\draw (-1,-4) circle[radius=4pt]; 
\node[circle,inner sep=4pt,fill=black,label= above left:{$x_1$}] at (2,2) {};
\node[circle,inner sep=4pt,fill=black,label= above left:{$x_2$}] at (-2.1,-1) {};
\node[circle,inner sep=4pt,fill=black,label= above left:{$x_3$}] at (-1,-3.15) {};
\node (a) at (0.5,-4.8) {$D(x_1)$= 1/6, $D(x_2)$= 0, $D(x_3)$= 1/6};

    \end{scope}
\end{tikzpicture}

\begin{tikzpicture}[label distance=-1mm]
    \begin{scope}[thick,font=\large]

\fill[ggplot_green!60]      (0,4) -| (4,0) -- cycle;
\fill[ggplot_orange!60]    (-0.5,4)   -| (-3,-4) -- cycle;
\fill[ggplot_blue!60]    (-3,-4)   -| (4,-1) -- cycle;
\draw[ggplot_green] (0,4) -- (4,0) node [above left] {};
\draw[ggplot_orange] (-0.5,4) -- (-3,-4) node [above left] {};
\draw[ggplot_blue] (-3,-4) -- (4,-1) node [above left] {};

    \draw [->] (-3,0) -- (4.2,0) node [above left] {};
    \draw [->] (0,-4) -- (0,4.2) node [below right] {};

\draw (3,3) circle[radius=4pt];
\draw (-1,1) circle[radius=4pt];
\draw (1,-2) circle[radius=4pt];
\draw (-1,-2) circle[radius=4pt]; 
\draw (-2,-3) circle[radius=4pt]; 
\draw (-1,-4) circle[radius=4pt]; 
\draw (2,-3) circle[radius=4pt]; 
\node[circle,inner sep=4pt,fill=black,label= above left:{$x_1$}] at (2,2) {};
\node[circle,inner sep=4pt,fill=black,label= above left:{$x_2$}] at (-2.1,-1) {};
\node[circle,inner sep=4pt,fill=black,label= above left:{$x_3$}] at (-1,-3.15) {};
\node (a) at (0.5,-4.8) {$D(x_1)= 1/7$, $D(x_2)= 0$, $D(x_3)= 2/7$};

    \end{scope}
\end{tikzpicture}
}
\end{center}
\caption{As $Q^{(m)}$ approaches $Q$, the depth pattern of $(x_1,x_2,x_3)$ changes: $\Pi_{Q^{(3)}}(x_1,x_2,x_3)=(1,3,3)$, $\Pi_{Q^{(6)}}(x_1,x_2,x_3)=(2,3,2)$ and $\Pi_{Q^{(7)}}(x_1,x_2,x_3)=(2,3,1)$. The black dots are the values of $x_i$ while the white dots denote the point-masses of $Q^{(m)}$.}
\label{fig:empirical_depth}
\end{figure}

A first step is hence to show that for $x\in \R^d$
\begin{align*}
D_{Q^{(m)}}(x)\longrightarrow D_Q(x) \ \text{a.s.}
\end{align*}
The $x$ can be seen as $(X_{i+1}(\omega), ..., X_{i+p}(\omega))$ for fixed $\omega$ and $i$.
Secondly, we let the number of observations of $X$ tend to infinity. 
We state the latter (our main result) first and provide the former as a proposition, subsequently.

\begin{theorem} \label{thm:consistency(C}
Suppose that  $X=(X_j)_{j\in \mathbb{N}}$  and  $Y=(Y_j)_{j\in \mathbb{N}}$  are two stationary ergodic time series defined on the same probability space $\left(\Omega, \mathcal{F}, P\right)$ and with values in $\mathbb{R}^d$. Assume that $Q$ separates $X$ by depth. 
Let $Q^{(m)}$  denote the empirical distribution of $Y_1, \ldots, Y_m$. Then, $\hat{q}_{n,m, Q}(\pi)$ is a consistent estimator of $p_Q(\pi):=P\pr{\Pi_Q(X_1,\ldots,X_{p})=\pi}$ in the following sense: For almost every $\omega$ there exists a subsequence $(m_n)_{n\in \N}$ (depending on this $\omega$), such that
\[
 \lim_{n\rightarrow \infty }\hat{q}_{n, m_n, Q}(\pi)=p_Q(\pi).
\]
\end{theorem}

\begin{proof}
Fix an $\omega\in\Omega$ for which \eqref{firstcon} holds and $n\in\N$.
Since $Q$ separates  $X$ by depth, there exists an $\varepsilon_n$
such that the distance between the depths of each two data-points $X_i(\omega)$,$X_j(\omega)$, $1\leq i< j\leq n+p$ is bigger than $2\varepsilon_n$.
By Proposition \ref{prop:conv_of_depth} there exists an $m_n \in \mathbb{N}$, such that for all $m\geq m_n$, 
$\left|D_{Q^{(m)}}(X_i(\omega))-D_Q(X_i(\omega))\right| < \varepsilon_n$ for all $i\in\{1, \ldots, n+p\}$ and, therefore,
\begin{align*}
\Pi_{Q^{(m)}}(X_{t+1}(\omega), \ldots, X_{t+p}(\omega))= \Pi_{Q}(X_{t+1}(\omega), \ldots, X_{t+p}(\omega)). 
\end{align*}
Hence finally for all $m\geq m_n$ and $i\in\{1, \ldots, n+p\}$
 \begin{align*}
1_{\left\{\Pi_{Q^{(m)}}(X_{t+1}(\omega),X_{t+1}(\omega), \ldots, X_{t+p}(\omega))=\pi\right\}}=  1_{\left\{\Pi_{Q}(X_{t+1}(\omega),X_{t+1}(\omega), \ldots, X_{t+p}(\omega))=\pi\right\}}.
 \end{align*}
 Therefore, $\lim_{n\rightarrow \infty}\hat{q}_{n,m_n, Q}(\pi)=p_Q(\pi)$ by Proposition \ref{prop_1:consistent}.
\end{proof}



If more and more data points of $X$ are observed, they can be arbitrary close to each other. Hence one needs more observations of $Y$ in order to distinguish them by depth. The dependence between $n$ and $m$ is inevitable. Key to the proof of Theorem \ref{thm:consistency(C} is to show that for every point in $\R^d$ the depth converges as the reference measures $Q^{(n)}$ approaches $Q$: 

\begin{proposition}\label{prop:conv_of_depth}
Let  $(Y_j)_{j\in\N}$ be  a  stationary ergodic time series with values in $\mathbb{R}^d$ having  marginal distribution $Q$.
Let $Q^{(n)}$
denote the empirical distribution of $Y_1, \ldots, Y_n$.
Then, it holds that
\begin{align*}
\sup\limits_{x\in \mathbb{R}^d}\left|D_{Q^{(n)}}(x)- D_{Q}(x)\right|\longrightarrow 0 \ a.s.
\end{align*}
\end{proposition}

\begin{proof}
Let $(Y_j)_{j\in\N}$ be a stationary ergodic sequence.
We make use of the fact that ergodicity is invariant under measurable transformations $f$, which is considered \enquote{mathematical folklore}, but, nonetheless, needs to be rigorously proved  for completeness (see Lemma \ref{lem:ergodic}). Particularly, we chose $f:\mathbb{R}^d\longrightarrow\mathbb{R}$, $z\mapsto 1_{\left\{\left(-\infty, y_1\right]\times \cdots \times\left(-\infty, y_d\right]\right\}}(z)$.
It follows from Birkhoff's ergodic theorem that 
 the empirical distribution function corresponding to $Q^{(n)}(\omega)=\frac{1}{n}\sum_{j=1}^n\delta_{Y_j(\omega)}$ converges, i.e.,
\begin{align*}
\frac{1}{n}\sum\limits_{j=1}^n1_{\left\{Y_j\in \left(-\infty, y_1\right]\times \cdots \times  \left(-\infty, y_d\right]\right\}}\longrightarrow Q\left(\left(-\infty, y_1\right]\times \cdots \times \left(-\infty, y_d\right] \right)=F_Q(y) \ a.s.
\end{align*}
for all $y=(y_1, \ldots, y_d)\in \mathbb{R}^d$.
Accordingly,  the distribution $Q^{(n)}$ converges weakly to $Q$ almost surely. 
For  $x \in \mathbb{R}^d$ and $\varphi \in \mathcal{S}^{d-1}$ we define $H_{x, \varphi}:=\left\{z\in \mathbb{R}^d|(z-x)^T\varphi\geq 0\right\}$.
According to the Portmanteau theorem
\begin{align} \label{qnconvergence}
Q^{(n)}\left(H_{x, \varphi}\right)\longrightarrow Q\left(H_{x, \varphi}\right),
\end{align}
 since the boundary of $H_{x, \varphi}$ is a set with measure $0$ due to continuity of $F_Q$.
For a proof, however, it has to be shown that 
    \begin{align}\label{qnconvergence_uni}
 \sup\limits_{x \in \mathbb{R}^d, \varphi \in \mathcal{S}^{d-1}}\left|Q^{(n)}\left(H_{x, \varphi}\right)- Q\left(H_{x, \varphi}\right)\right|\longrightarrow 0 \ \text{a.s.}
    \end{align}
    For this note that according to  Theorem 1 in \cite{elker1979glivenko}
\eqref{qnconvergence} implies \eqref{qnconvergence_uni} for  stationary ergodic random observations $(Y_j)_{j\in \mathbb{N}}$ with values in $\mathbb{R}^d$
if $\{H_{x, \varphi}: x\in \mathbb{R}^d, \varphi \in \mathcal{S}^{d-1}\}$ constitutes a so-called ideal uniformity class.
Moreover, Example 11 in \cite{elker1979glivenko} shows that the closed halfspaces in $\mathbb{R}^d$ constitute such an ideal uniformity class.
\end{proof}



\section{Data analysis: animal movement}\label{sec:data}

For an application of depth patterns we analyzed seal pub movement data provided within the scope of tracking studies considered in \cite{nagel2021movement}.
The corresponding data set contains hourly GPS data from 66 Antarctic fur seal pups tracked from birth until moulting.
The individuals were selected at random from two breeding colonies on Bird
Island, South Georgia: the Special Study Beach (SSB) and Freshwater Beach (FWB), which are separated by around 200m. Due to their close proximity, these colonies experience comparable climatic conditions
and offer comparable habitats. Despite these similarities, the seal pub population density  is higher at SSB  resulting in  different movement dynamics at the two colonies.
A higher density of animals at SSB results in a more static system. At the same time, low density breeding aggregations are more likely to be chased by avian predators, such that pubs  at FWB 
have a higher risk of  
 being killed.

Our analysis focuses on  the seal pub population located at  FWB, a total of 21 individuals. 
Out of these 21 individuals six   died within the period of recorded measurements
and were, therefore, disregarded in the analysis.
Accordingly, we based our analysis on the remaining 15 individuals with ID numbers
 C1, C5, C6, C7, C8, C9, C10, C12, C13, C17, C18, C19, C21, C22, C25.
The deployment duration of these animals ranged from a minimum of 498  hours (approximately 21 days) to a maximum of  1287 hours  (approximately 54 days), while
on average individuals travelled  approximately 50 meters per hour.
\begin{figure}[h] 
\begin{center}
\includegraphics[width=8cm]{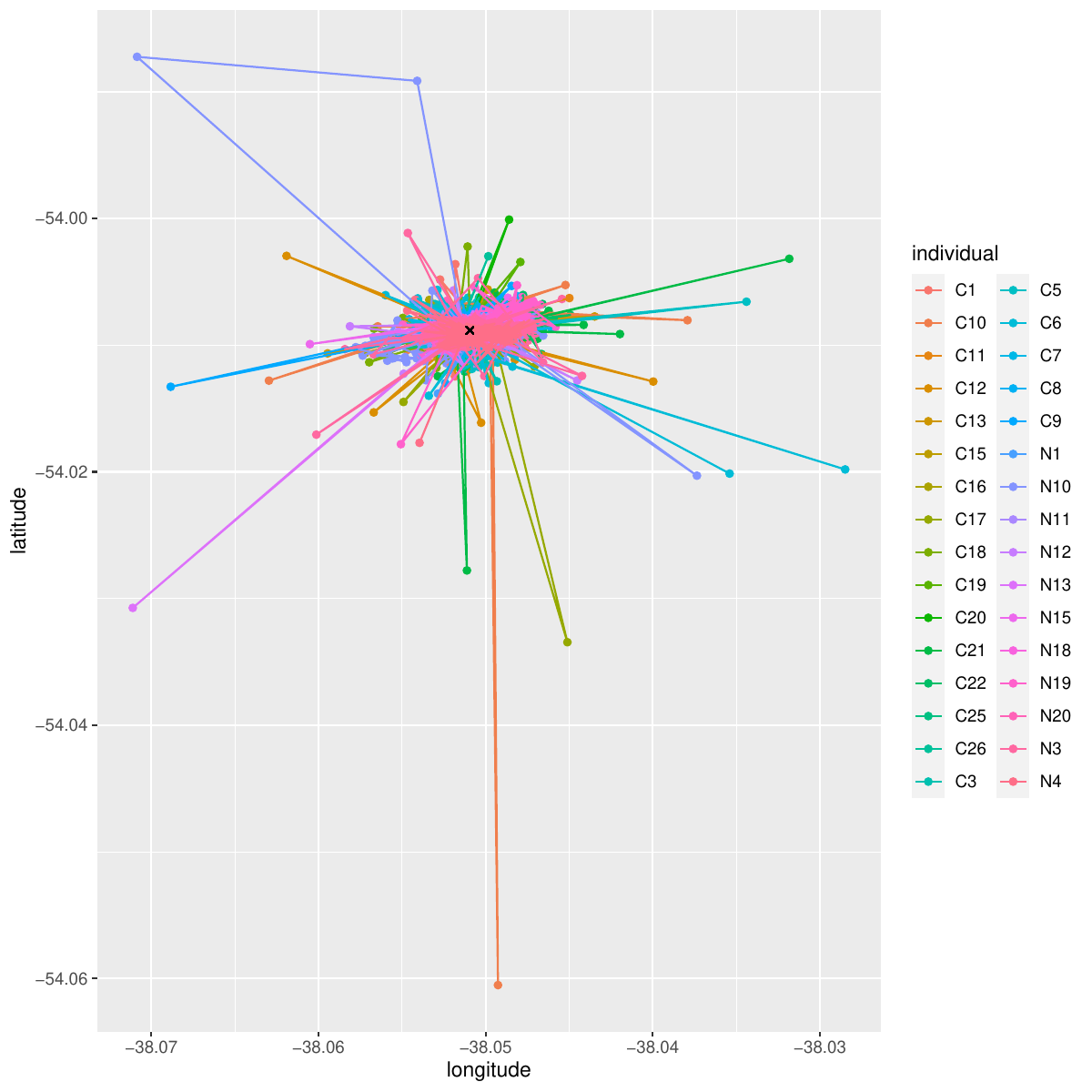}
\caption{Motion paths of 15 individual seal pubs at Freshwater Beach (FWB).}
\label{fig:fwb}
\end{center}
\end{figure}
In general, pup movement showed a star-like pattern characterized by directed exploration within a relatively small convex 
area around a central location; see Figure \ref{fig:fwb}.

\begin{table}[]
    \centering
    \begin{tabular}{ll|cccccc}
    ID   & {$n\backslash\pi$}  & (0, 1, 2) & (0, 2, 1) & (2, 0, 1) & (2, 1, 0) & (1, 2, 0) & (1, 0, 2)\\
    \hline
      C1  &
      1251 & 0.1927126 & 0.1392713 & 0.1489879 & 0.1983806 & 0.1651822 & 0.1554656\\
      C5 & 1164 &  0.1755459 & 0.1589520 & 0.1458515 & 0.2052402 & 0.1510917 & 0.1633188\\
     C6 & 1128 & 0.1931408 & 0.1525271 & 0.1398917 & 0.1958484 & 0.1534296 & 0.1651625\\
    C7 & 1287 & 0.1804571 & 0.1631206 & 0.1623325 & 0.1946414 & 0.1489362 & 0.1505122\\
    C8 & 1244 & 0.1763265 & 0.1510204  & 0.1477551 & 0.1893878 & 0.1665306 & 0.1689796\\
    C9 & 922 & 0.1743929 & 0.1721854 & 0.1534216 & 0.1942605 & 0.1434879 & 0.1622517\\
    C10  & 758 & 0.1878378 & 0.1608108 & 0.1459459 & 0.1689189 & 0.1621622 & 0.1743243\\
    C12 & 1271 & 0.1600318 & 0.1711783 & 0.1536624 & 0.1910828 & 0.1560510 & 0.1679936\\
    C13 & 1097 & 0.1922006 & 0.1504178 & 0.1662024 & 0.1708449 & 0.1680594 & 0.1522748\\
    C17 & 1260 & 0.1771845 & 0.1521036 & 0.1504854 & 0.1820388 & 0.1674757 & 0.1707120\\
    C18 & 1228 & 0.1679074 & 0.1728701 & 0.1720430 & 0.1695616 & 0.1579818 & 0.1596361\\
    C19 & 498 & 0.2108559 & 0.1503132 & 0.1691023 & 0.1774530 & 0.1544885 & 0.1377871\\
    C21 & 1240 & 0.1797386 & 0.1503268 & 0.1748366 & 0.1683007 & 0.1764706 & 0.1503268\\
    C22 & 1260 & 0.1771845 & 0.1521036 & 0.1504854 & 0.1820388 & 0.1674757 & 0.1707120\\ 
    C23 & 908 & 0.1629464 & 0.1808036 & 0.1741071 & 0.1941964 & 0.1417411 & 0.1462054
    \end{tabular}
    \caption{Frequencies of the ordinal patterns of order 3 for the 15 considered seal pubs at Freshwater Beach (FWB). $n$ denotes the number of observations available for each individual, ID their ID numbers. }
    \label{table:op_frequencies_seal_pubs}
\end{table}

Based on the  data, the goal of our analysis is the determination of a suitable model for a characterization of seal pub movement at Freshwater Beach (FWB).
For this, we characterize  the pub movement  in terms of steps taken between successive locations at regular time intervals (in our case time intervals of one hour), while steps are characterized by the distance between successive locations (step lengths) and changes in direction (turning angles) assuming that these are independent.

Due to its generality 
often the gamma distribution (characterized  by a shape parameter $k$ and a scale parameter 
$\theta$) is used to model step lengths of animal movement; see \cite{avgar2016integrated}.
Maximum-Likelihood estimation in \verb$R$ via the command  \verb$egamma$   on the basis of movement data from all 15 considered animals at Freshwater Beach  resulted in the estimates 
$\hat{k}=1$ and $\hat{\theta}=48.44$.
Since a gamma distribution with shape parameter $k=1$ and scale parameter $\theta$ corresponds to an exponential distribution with rate parameter $\lambda=\frac{1}{\theta}$, we  model the step length distribution by an 
exponential distribution with rate parameter $\lambda=\frac{1}{\hat{\theta}}=0.02$.
For an illustration of the realized step lengths see Figure \ref{fig:steps}.
\begin{figure}[h] 
\begin{center} 
\includegraphics[width=8cm]{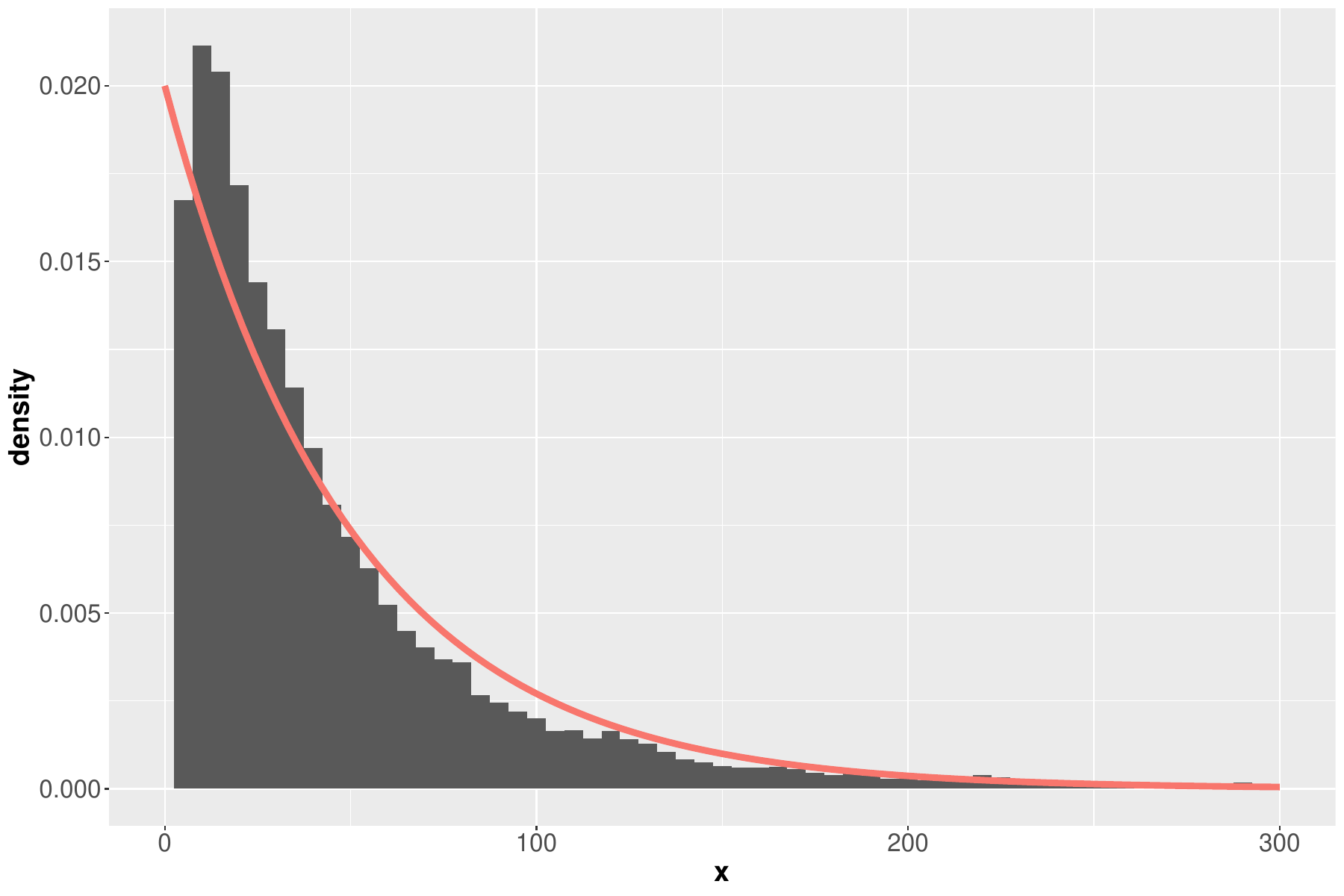}
\caption{Histogram  of step lengths of 15 individual seal pubs at Freshwater Beach (FWB) and density of the exponential distribution with parameter $\lambda=\frac{1}{\hat{\theta}}$, where $\hat{\theta}=48.44$ derives from Maximum-Likelihood estimation based on the step lengths of all 15 individuals.}
\label{fig:steps}
\end{center}
\end{figure}

A two-dimensional random walk model aligning with this modeling approach for step lengths corresponds to turning angles drawn  independently from a uniform distribution on $[0, 2\pi]$.
Accordingly, we choose step lengths generated independently from an exponential distribution with rate parameter $\lambda=0.02$ and turning angles generated independently from a uniform distribution on $[0, 2\pi]$.

The ordinal pattern distribution of a corresponding random walk model is then approximated through simulations by an average of pattern frequencies of 1000 repeated simulations of movement paths, each  consisting of 1000 steps; see Table \ref{table:op_frequencies_rw}.
\begin{table}[]
    \centering
    \begin{tabular}{cccccc} 
(0, 1, 2) & (0, 2, 1) & (2, 0, 1) & (2, 1, 0) & (1, 2, 0) & (1, 0, 2)\\
    \hline
     0.2482 &
 0.126 &
 0.126 &
 0.249 &
 0.126 &
 0.126
     \\
    \end{tabular}
    \caption{Frequencies of the ordinal patterns of order 3 for a random walk  consisting of 1000 steps with steplengths generated independently from an exponential distribution with rate parameter $\lambda=0.02$ and turning angles generated independently from a uniform distribution on $[0, 2\pi]$. The frequencies correspond to an average from 1000 repetitions.}
    \label{table:op_frequencies_rw}
\end{table}
A comparison of the pattern distributions of the individuals at FWB reported in Table \ref{table:op_frequencies_seal_pubs} and the pattern distribution  reported in Table 
\ref{table:op_frequencies_rw} clearly indicates that the seal pub movement at FWB does not follow a random walk.

As an alternative to the simple random walk model, we propose to describe the seal pub movement by a
{\em
biased (anti)persistent random walk}. For this, we assume that driving forces for the seal pubs' movement are a bearing towards a point of interest (the \enquote{center of animal movement}), as well as dependence between the next and previous steps in the movement.
More precisely, we model the expectation of the turning angle $a(t)$ at each step $t$
as a convex combination of the previous turning angle $a(t-1)$  and a bias $b(t)$ 
towards the center $c=(c_1, c_2)$  of  animal movement.
For this, we determine the bearing $b(t)$ from the current position $x(t)=(x_1(t), x_2(t))$ of the animal to the center as the angle
of a right-angled triangle with adjacent side $\Delta x_1(t)=c_1-x_1(t)$
and opposite side $\Delta x_2(t)=c_2-x_2(t)$, i.e. 
$b(t)=\operatorname{atan2}\left(\Delta x_2(t), \Delta x_1(t)\right)$, where $\operatorname{atan2}(y, x)$ is the angle measure (in radians, with 
$-\pi <\theta \leq \pi$) between the positive $x$-axis and the ray from the origin to the point $(x, y)$ in the Cartesian plane.

\begin{figure}[!h]
\begin{center}
\begin{tikzpicture}[remember picture]
\tikzset{->-/.style={decoration={
 markings,
 mark=at position .5 with {\arrow{>}}},postaction={decorate}}}
\coordinate (C) at (0,0); 
\node[below left] at (C) {\scriptsize $(c_1, c_2)$};
\coordinate[circle, fill,inner sep=1.5pt]  (A) at (5,2.5);
\node at (3.6,2.4) {\scriptsize$(x_1(t), x_2(t))$};
\coordinate[circle, fill,inner sep=1.5pt] (B) at (-1,4); 
\node[above right] at (B) {\scriptsize$(x_1(t-1), x_2(t-1))$};
\draw[thick] (C) -- (5, 0) node[midway,below] {\scriptsize $\Delta x_1(t)$};
\draw[thick] (A) -- (5,0) node[midway,right] {\scriptsize $\Delta x_2(t)$};
\draw[dashed] (C) -- (A);
\draw[->,thick] (B) --  (A);
\draw[dashed] (A) -- (11, 1);
\draw[dotted] (A) -- (11, 2.5);
\draw (5,0) -- ++(-0.3,0) -- ++(0,0.3) -- ++(0.3,0);
\draw[->,thick, ggplot_blue]  (A) -- (8.5, 1);
\coordinate[circle, fill, inner sep=1.5pt, ggplot_blue] (D) at (8.5, 1) ; 
\node[below, ggplot_blue] at (8.5, 1) {\scriptsize$(x_1(t+1), x_2(t+1))$};
\draw[thick,->] (9.5, 2.5) arc[start angle=0,end angle=-40,radius=1.5];
\node at (8.5,2) {\scriptsize $a(t-1)$};
\draw[thick,->] (1.8,0) arc[start angle=0,end angle=40,radius=1.2];
\node at (1.2,0.25) {\scriptsize $b(t)$};
\draw[->, thick, ggplot_orange] (A) -- (3.211146, 1.605572);
\draw[thick, ggplot_green] (A) -- (5, 1.405572) node[midway,right] {\tiny $\sin(b(t))$};
\draw[thick, ggplot_green] (5, 1.605572)-- (3.211146, 1.605572)
node[midway,below] {\tiny $\cos(b(t))$};
\draw[thick, ggplot_green] (A) -- (6.940285, 2.5) node[midway, above] {\tiny $\cos(a(t-1))$};
\draw[thick, ggplot_green] (6.940285, 2.5) --  (6.940285, 2.014929)
node[midway,right] {\tiny $\sin(a(t-1))$};
\draw[->, thick, ggplot_orange] (A) -- (6.940285, 2.014929);
\end{tikzpicture}
\caption{Illustration of a step in a biased (anti)persistent random walk.
In blue the step from the current to the next location, i.e. from $(x_1(t), x_2(t))$ to $(x_1(t+1), x_2(t+1))$.
The red arrows indicate the bearing to the center $(c_1, c_2)$ and in direction of the previous turning angel $a(t-1)$ standardized to a steplength of 1. In green the corresponding steps  taken in direction of $x$- and $y$-axis. }
\end{center}
\end{figure}

The turning angle $a(t)$ is then 
 chosen from a von Mises distribution with location parameter $\mu$ which is calculated as a `weighted mean' of the 
 turning angle and the bearing to the center. The concentration parameter $\kappa$ is assumed to be constant over time. 
 In fact,  the von Mises distribution is often the model of choice when describing  (anti)persistent random walks; see
\cite{wu2000modelling}.
For an illustration of the von Mises distribution with different values of the parameter $\kappa$ see Figure \ref{fig:vonmises}.
\begin{figure}[h] 
\begin{center}
\includegraphics[width=8cm]{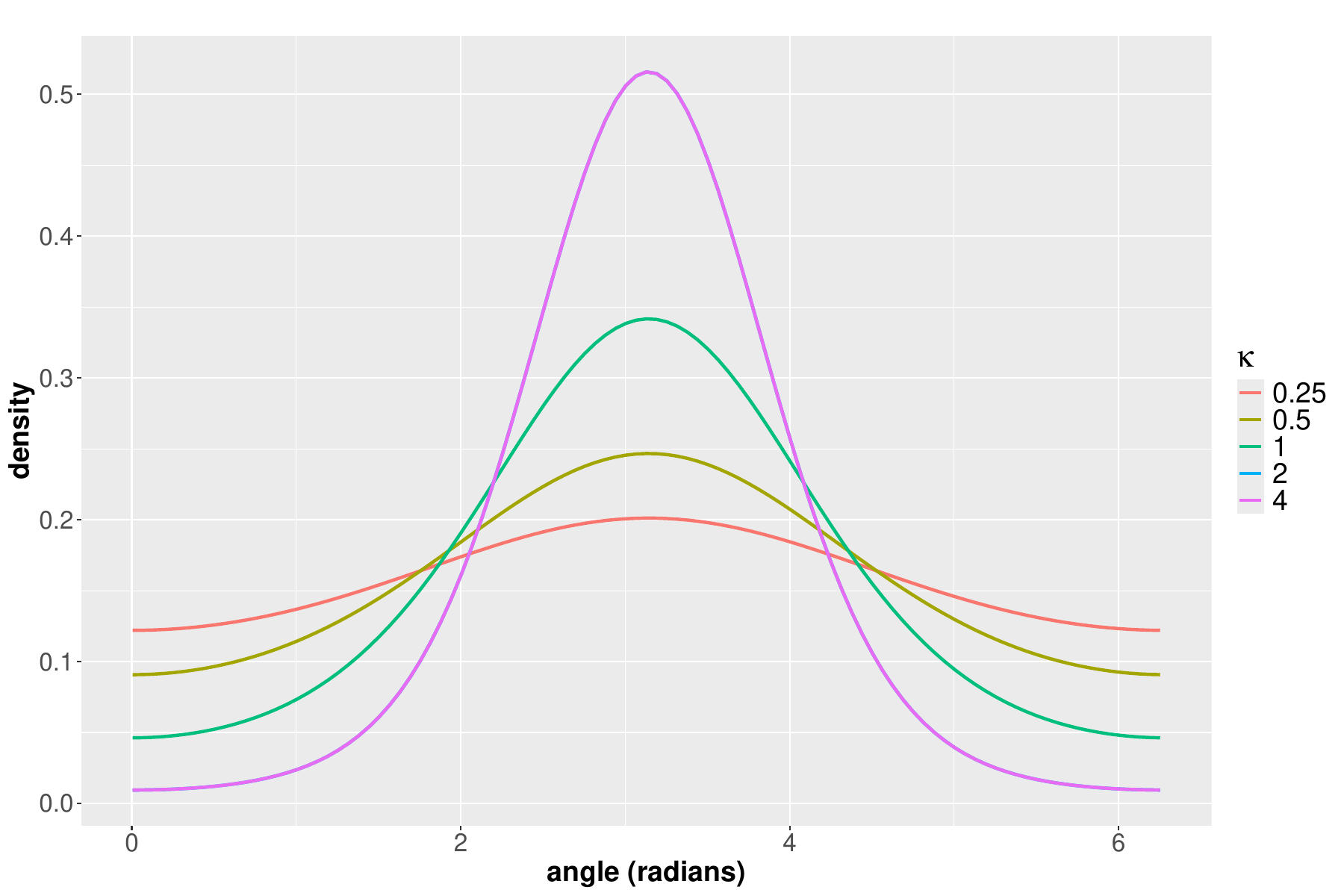}
\caption{Density functions of the von Mises distribution for different values of the concentration parameter $\kappa$.}
\label{fig:vonmises}
\end{center}
\end{figure}
For the specific situation, we choose $\mu= \operatorname{atan2}\left(\Delta_2(t)/\Delta_1(t)\right)$, where 
\begin{align*}
    &\Delta_1(t) =r(1-\beta) \cos(a(t-1))+\beta \cos(b(t))\\
     &\Delta_2(t) =r(1-\beta) \sin(a(t-1))+\beta \sin(b(t)),
\end{align*}
for some $\beta\in [0, 1]$ and $r\in\{-1, 1\}$.
 The choice $r=1$ corresponds to persistent movement, i.e. to a turning angle that does not deviate much from the previous turning angle,
while  $r=-1$  corresponds to antipersistent movement, i.e. to a turning angle that is expected to change by approximately $180^{\circ}$.

\vspace{10mm}
\includegraphics[width=\textwidth]{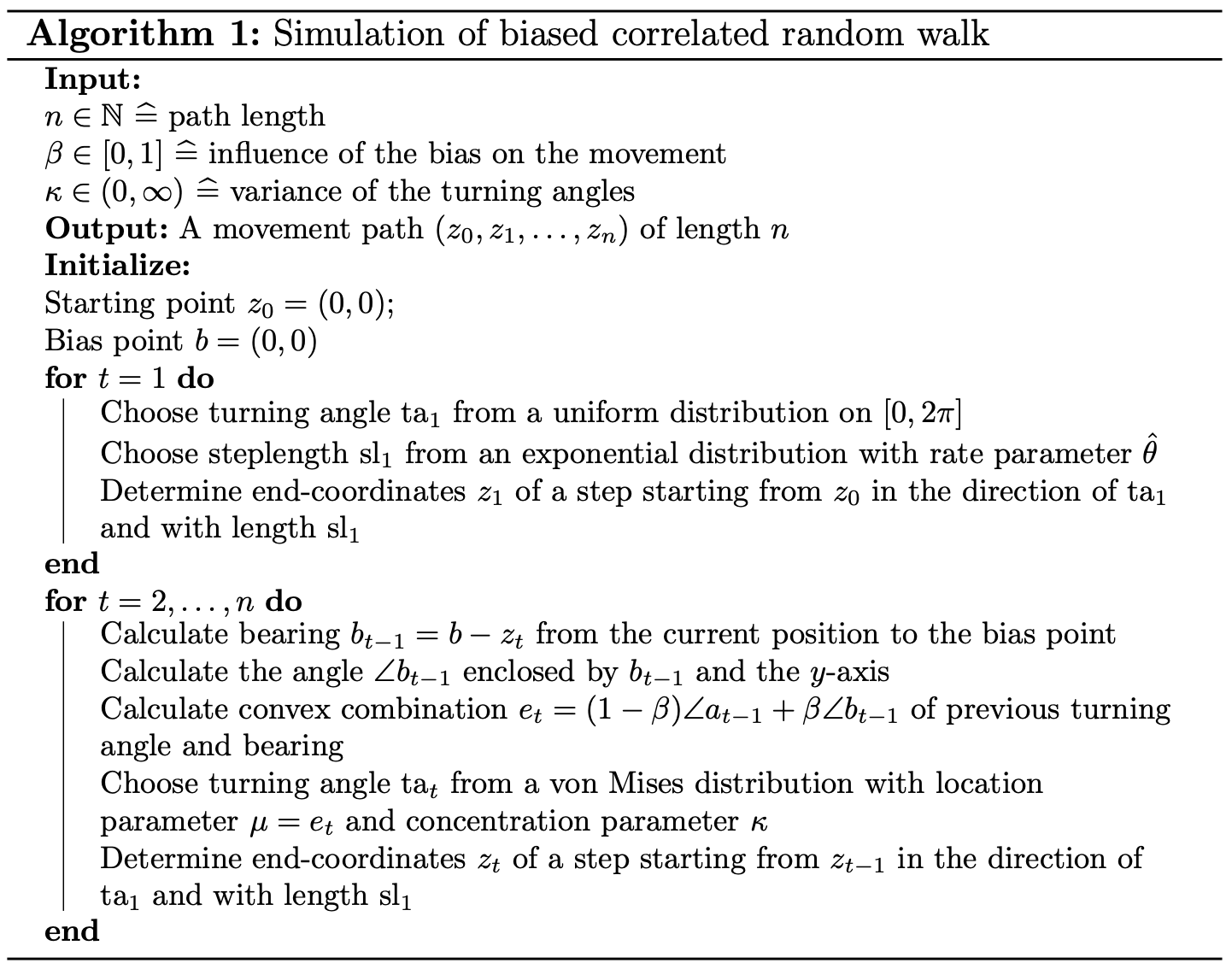}


For each of the parameter combinations in the tables below we simulated 1000 movement paths.
For each path we have computed the frequencies of the individual ordinal pattern and then averaged the pattern frequencies over all 1000 repetitions, i.e.,
for each parameter combination, we have obtained
six pattern frequencies $q=(q_1, \ldots, q_6)$.
Moreover, for each of the 15 animals considered 
let $p^i=(p_1^i, \ldots, p_6^i)$, $i=1, \ldots, 15$, denote the pattern frequencies of the i-th animal.
In the tables below we report the sum of Euclidean distances $\sum_{i=1}^{15}\|p^i-q\|_2$.

\begin{table}[h]
\centering
\begin{tabular}{ll|cccccc}
\\
 &$\kappa \backslash\beta$  & 0 & 0.2 & 0.4 & 0.6 & 0.8 & 1 \\ 
 \cline{2-8} 
&$ 0.25$  &  2.276 & 2.281 & 2.22 & 1.945 & 1.822 & 1.759 \\ 
&$ 0.50$  &  2.768 & 2.761 & 2.648 & 2.110 & 1.929 & 1.850 \\ 
\smash{\rotatebox[origin=c]{90}{\text{$r=1$}}}  &$ 1$  &  3.594 & 3.627 & 3.371 & 2.454 & 2.287 & 2.200 \\
&$ 2$  &  4.592 & 4.744 & 4.140 & 2.885 & 2.859 & 2.804 \\ 
&$ 4$  &  5.097 & 5.526 & 4.565 & 3.150 & 3.222 & 3.212 \\ 
 \cline{2-8} 
&$0.25$  &  1.281 & 1.296 & 1.332 & 1.609 & 1.701 & 1.768 \\ 
&$0.50$  &  0.829 & 0.837 & 0.945 & 1.541 & 1.754 & 1.849 \\ 
\smash{\rotatebox[origin=c]{90}{\text{$r=-1$}}}&$1$  &  0.409 & 0.436 & 0.647 & 1.757 & 2.101 & 2.211 \\ 
&$2$  &  1.186 & 1.246 & 1.246 & 2.328 & 2.704 & 2.800 \\ 
&$4$  &  1.917 & 1.968 & 1.763 & 2.705 & 3.130 & 3.203 \\ 
\end{tabular}
\caption{
Summarized squared differences of the individuals'
ordinal pattern distributions from simulated pattern distributions generated according to the parameters $r, \kappa$, and $\beta$.
}
\label{table:function_values}
\end{table}

Most notably, for $r=1$ the values reported in Table \ref{table:function_values} are significantly higher than those for $r=-1$. We conclude that seal pub movement can be better described by biased antipersistent random walks than by  biased  persistent random walks.
The smallest values in  Table \ref{table:function_values} correspond to small values of $\beta=0$ and $\rho=1$.
This suggests that for  seal pub movement the
dependence between turning angles has a stronger influence on the animals' movement pattern than the 
bearing towards the center.

\section{Discussion}
\label{sec:conc}

 The present article establishes the concept of depth patterns in order to analyze movement in the plane on an ordinal scale. Due to a lack of a natural ordering in $\mathbb{R}^2$, we  consider  Tukey's halfspace depth as a tool to determine a center-outward ordering  of  datapoints. This allows us to keep track of  movement towards or away from the center of a certain region. 
Depending on the relationship between the underlying time series and the reference measure we  consider four different settings and  prove limit theorems within each setting. 
Whenever using ordinal approaches, this has several benefits but also drawbacks. On the one hand, our method is stable under small perturbations or measurement errors. On the other hand, due to basing our analysis on ordinal information only, we lose  metric information. Nonetheless, the results can be interpreted in a natural way and our approach is flexible and easy to implement.  
We  emphasize the applicability of our approach by analyzing the movement of seal pups. One can think of various other areas of applications like, for example,  rainfall radar data. 
 \cite{neu_coh_tam_21} suggest to use ordinal patterns for prediction in natural environments. They  proclaim that these patterns could explain, how `a myopic organism can make short-term predictions given limited data availability and cognitive resources'. Although the idea seems interesting and may be useful for future research, the 
considerations of  \cite{neu_coh_tam_21} in this direction stay vague.
In fact, the practical 
 example considered in that article does not seem to fit into the envisioned framework as it aims at short-term prediction of  FinTech data, i.e., BitCoin prices. 
Complementing the vision of \cite{neu_coh_tam_21} we conjecture  that  natural environments like myopic organisms are predestined for  depth pattern-based analyses (which, naturally, had not  been available at the time  \cite{neu_coh_tam_21} was published). Modeling chemotaxis-driven bacteria movement can be considered as a natural example in this regard. 
 With the present article, we  close a corresponding gap and our method might be useful for modeling bacteria movement where the individuals are using chemotaxis.

\appendix

\section{Ergodicity}

\begin{definition}
    Let $(\Omega, \mathcal{F}, P, T)$ be a measure preserving dynamical system.
    A measurable set $A\subset\Omega$ is called $T$-invariant if $T^{-1}A=A$.
    Then
    \begin{align*}
\mathcal{I}:=\left\{A:T^{-1}(A)=A\right\}        
    \end{align*}
    denotes the $\sigma$-field of invariant sets.
    $T$ is called ergodic if $\mathcal{I}$ is trivial, i.e., $P(A)\in \{0, 1\}$ for all $A\in \mathcal{I}$.
    A stationary time series $Z=(Z_t)_{t\in \mathbb{N}}$ with values in $S$ is called ergodic if the shift on 
    $(S^{\mathbb{N}}, \mathcal{S}^{\mathbb{N}}, P_Z)$ is ergodic.
\end{definition}

\begin{lemma}\label{lem:ergodic}
Suppose that $(X_j)_{j\in \N}$ is a stationary ergodic time series with state space $S$ and $f:S^{\mathbb{N}}\longrightarrow \mathbb{R}$ a measurable function. 
Then $(Y_j)_{j\in \N}$ with $Y_i=f(X_i, X_{i+1}, \ldots)$ is stationary ergodic as well. 
\end{lemma}
\begin{proof}
It is a well-known fact that measurable transformations preserve strict stationarity.
Therefore, it suffices to prove ergodicity of $(Y_j)_{j\in \mathbb{N}}$.
For this, let $T$ denote the shift operator and let $A\in \mathcal{I}_Y$, i.e.,
\begin{align*}
    y=(y_0, y_1, \ldots )\in A\Leftrightarrow (y_1, y_2, \ldots)\in A.
\end{align*}
Define
\begin{align*}
B:=\{x:(f(x_0, x_1, \ldots), f(x_1, x_2, \ldots), \ldots )\in A\}.
\end{align*}
Then, it holds that 
$P_X(B)=P(X\in B)=P(Y\in A)=P_Y(A)$.
Moreover, $B$ is shift invariant as well since
\begin{align*}
    x\in B   &\Leftrightarrow (f(x_0, x_1, \ldots), f(x_1, x_2, \ldots), \ldots )\in A\\
    &\Leftrightarrow (f(x_1, x_2, \ldots), f(x_2, x_3, \ldots), \ldots )\in A\\
        &\Leftrightarrow (x_1, x_2, \ldots)=T(x
        )\in B.
\end{align*}
As a result, since $P_X(B) \in \{0, 1\}$ for all $B\in \mathcal{I}_X$, $f$ is ergodic.
\end{proof}






\section{A geometric proof of Theorem \ref{thm:consistency(C}}

This section presents an alternative approach for a proof of  Theorem \ref{thm:consistency(C}, which does not use the deep results of \cite{elker1979glivenko}, as supplementary material. Instead, we are using solely elementary geometry to present a proof in dimension $d=2$. 

For a proof of Theorem \ref{thm:consistency(C} it is enough to have the following pointwise version of Proposition \ref{prop:conv_of_depth}. 

\begin{proposition}\label{prop:conv_of_depth_point}
Let $x\in \mathbb{R}^2$  and let  $(Y_j)_{j\in\N}$ be  a  stationary ergodic time series with values in $\mathbb{R}^2$ having marginal distribution  $Q$.
Let $Q^{(n)}$
denote the empirical distribution of $Y_1, \ldots, Y_n$.
Then, it holds that
\begin{align*}
D_{Q^{(n)}}(x)\longrightarrow D_{Q}(x) \ a.s.
\end{align*}
\end{proposition}

\begin{proof}
Let $(Y_j)_{j\in\N}$ be a stationary ergodic sequence.
We make use of the fact that ergodicity is invariant under measurable transformations $f$, (see Lemma 3). Particularly, we chose $f:\mathbb{R}^2\longrightarrow\mathbb{R}$, $z\mapsto 1_{\left\{\left(-\infty, y_1\right]\times \left(-\infty, y_2\right]\right\}}(z)$.
It follows from Birkhoff's ergodic theorem that 
 the empirical distribution function corresponding to $Q^{(n)}(\omega)=\frac{1}{n}\sum_{j=1}^n\delta_{Y_j(\omega)}$ converges, i.e.
\begin{align*}
\frac{1}{n}\sum\limits_{j=1}^n1_{\left\{Y_j\in \left(-\infty, y_1\right]\times \left(-\infty, y_2\right]\right\}}\longrightarrow Q\left(\left(-\infty, y_1\right]\times \left(-\infty, y_2\right] \right)=F_Q(y) \ a.s.
\end{align*}
for all $y=(y_1, y_2)\in \mathbb{R}^2$.
Accordingly,  the distribution $Q^{(n)}$ converges weakly to $Q$ almost surely. 
For fixed $x \in \mathbb{R}^2$  we define $a_n(\varphi):=Q^{(n)}\left(\left\{z\in \mathbb{R}^2|(z-x)^T\varphi\geq 0\right\}\right)$ and $a(\varphi):=Q\left(\left\{z\in \mathbb{R}^2|(z-x)^T\varphi\geq 0\right\}\right)$. 
According to the Portmanteau theorem, for fixed $\varphi$
\begin{align*} 
Q^{(n)}\left(\left\{z\in \mathbb{R}^2|(z-x)^T\varphi\geq 0\right\}\right)\longrightarrow Q\left(\left\{z\in \mathbb{R}^2|(z-x)^T\varphi\geq 0\right\}\right),
\end{align*}
that is, 
 $a_n$ converges \emph{pointwise} in $\varphi$, since the boundaries of the sets under consideration are  sets with measure $0$ due to continuity of $F_Q$.

The idea is now to separate the half planes $\left\{z\in \mathbb{R}^2|(z-x)^T\varphi\geq 0\right\}$ into finitely many subsets and to show  that the convergence on these is \emph{uniform} in $\varphi$.  
More explicitly, we write the half planes as  differences of two circular sectors plus a half line  including the initial point. 
We show uniform convergence for the values of $\varphi$ being contained in the left half plane of $\R^2$. The right half plane can be considered in the same (technical) way. 

We consider a parametrization of $\mathbb{R}^2$ in polar coordinates through
the bijection $W:\mathbb{R}^2\setminus\left\{(0, 0)\right\}\longrightarrow
(0, \infty)\times \left[0, 2\pi\right)$ defined by
\begin{align*}
(z_1, z_2) \mapsto \left(\sqrt{z_1^2+z_2^2}
\begin{cases}
\arctan\left(\frac{z_2}{z_1}\right) \ &\text{if $z_1>0$, $z_2\geq 0$}\\
\frac{\pi}{2} \ &\text{if $z_1=0$, $z_2> 0$}\\
\arctan\left(\frac{z_2}{z_1}\right) +\pi \ &\text{if $z_1<0$}\\
\frac{3\pi}{2} \ &\text{if $z_1=0$, $z_2< 0$}\\
\arctan\left(\frac{z_2}{z_1}\right) +2\pi \ &\text{if $z_1>0$, $z_2< 0$}
\end{cases}
\right).
\end{align*}
Given that $T_{-x}$ denotes a translation with $-x$, i.e. $T_{-x}(z)=z-x$, and $\pi_2:(0, \infty)\times [0, 2\pi)\longrightarrow [0, 2\pi)$ the projection on the second coordinate, that is, $\pi_2(x, y)=y$, we denote
$\Pi_x=\pi_2\circ W\circ T_{-x}:\mathbb{R}^2\setminus\{x\}\longrightarrow \left[0, 2\pi\right)$.

Consider $\varphi=(\varphi_1, \varphi_2)$. Depending on the value of $\varphi_1$, we consider the three cases $\varphi_1<0$, $\varphi_1>0$, and $\varphi_1=0$ separately.

Let $\varphi_1<0$ (case 1).
Note that for the halfspace we are interested in, it holds that
\begin{align*}
\left\{z\in \mathbb{R}^2|(z-x)^T\varphi\geq 0\right\}
=x+\left\{z\in \mathbb{R}^2\left|\right. z^{\top}\varphi\geq 0\right\}.
\end{align*}
Moreover, 
it holds that for some correspondingly chosen $t_{\varphi}$
\begin{align*}
&\left\{z\in \mathbb{R}^2\left|\right. z^{\top}\varphi\geq 0\right\}
=A_{t_{\varphi}+\frac{\pi}{2}}\setminus A_{t_{\varphi}-\frac{\pi}{2}}\cup \left\{z\in \mathbb{R}^2\left|\right. z_1\geq 0,z_2\geq 0,  z^{\top}\varphi= 0\right\},
\end{align*}
where for $t\in [0, 2\pi)$ the subset $A_t$ 
 of $\mathbb{R}^2$ is defined by 
\begin{align*}
A_{t}:=W^{-1}\left(\pi_2^{-1}\left[0, t\right]\right);
\end{align*} 
see Figure \ref{figure1} for a visualization and Lemma \ref{lemma:parametrization} below for a rigorous proof.
\begin{figure}[h] 
\begin{center}
  \begin{tikzpicture}[scale=3,cap=round,>=latex, thick, draw opacity=1]
                             \fill[magenta!10] plot coordinates {(-1.664101cm, -1.1094cm)(-1.664101cm, 1.664101cm)(1.664101cm, 1.664101cm)(1.664101cm, 1.1094cm)};
\fill[pattern color = violet!20, pattern=north west lines] plot coordinates {(1.664101cm, 0cm)(0,0)(-1.109401cm,  1.664101cm)(1.664101cm, 1.664101cm)};   
        \draw[->] (-1.664101cm,0cm) -- (1.664101cm,0cm) node[right] {$x$};
        \draw[->] (0cm,-1.664101cm) -- (0cm,1.664101cm) node[above] {$y$};
  \draw[-] (0cm,0cm) -- (-1.109401cm, 1.664101cm);
    \draw[violet, thick, ->] (0cm,0cm) -- (-0.5547002cm, 0.8320503cm);
        \draw[magenta, thick, -] (0cm,0cm) -- (1.664101cm, 1.1094cm);
                \draw[magenta, thick, -] (0cm,0cm) -- (-1.664101cm, -1.1094cm);       
   \draw[-] (0cm,0cm) -- (1.109401cm, -1.664101cm);
    \node  at (-0.2,0.5) {\textcolor{violet}{$\varphi$}};
        \node  at (0.6,1.2) {\textcolor{violet}{$A_{t_{\varphi}}$}};
                \node  at (-1,1) {\textcolor{magenta}{$\left\{z\in \mathbb{R}^2\left|\right. z^{\top}\varphi\geq 0\right\}$}};
        \draw[thick] (0cm,0cm) circle(1cm);
        \foreach \x/\xtext in {
            90/\frac{\pi}{2},
            180/\pi,
            270/\frac{3\pi}{2},
            360/2\pi}
                \draw (\x:0.85cm)  node[fill=white] {$\xtext$};
        \draw (-1.25cm,0cm) node[above=1pt] {$(-1,0)$}
              (1.25cm,0cm)  node[above=1pt] {$(1,0)$}
              (0cm,-1.25cm) node[fill=white] {$(0,-1)$}
              (0cm,1.25cm)  node[fill=white] {$(0,1)$};
                        \end{tikzpicture}

\caption{The half space $\left\{z\in \mathbb{R}^2\left|\right. z^{\top}\varphi\geq 0\right\}$ for $\varphi$ and the corresponding  circular sector $A_{t_{\varphi}}$. It follows that $\left\{z\in \mathbb{R}^2\left|\right. z^{\top}\varphi\geq 0\right\}
=A_{t_{\varphi}+\frac{\pi}{2}}\setminus A_{t_{\varphi}-\frac{\pi}{2}}\cup \left\{z\in \mathbb{R}^2\left|\right. z_1\geq 0,z_2\geq 0,  z^{\top}\varphi= 0\right\}$.}
\label{figure1}
\end{center}
\end{figure}

It then follows that
\begin{align*}
&\left\{z\in \mathbb{R}^2\left|\right. (z-x)^{\top}\varphi\geq 0\right\}\\
=&(\{x+A_{t_{\varphi}+\frac{\pi}{2}}\}\setminus \{x+A_{t_{\varphi}-\frac{\pi}{2}}\})   \cup \left\{z\in \mathbb{R}^2\left|\right. z_1-x_1\geq 0,z_2-x_2\geq 0,  z^{\top}\varphi= 0\right\},
\end{align*}
where the two sets are disjoint and $\{x+A_{t_{\varphi}-\frac{\pi}{2}}\}\subset \{x+A_{t_{\varphi}+\frac{\pi}{2}}\}$.
As a result, we have
\begin{align} \label{4terms}
&Q^{(n)}\left(\left\{z\in \mathbb{R}^2\left|\right. (z-x)^{\top}\varphi\geq 0\right\}\right)\\ \nonumber
= &Q^{(n)}\left(
\{x+A_{t_{\varphi}+\frac{\pi}{2}}\}\setminus \{x+A_{t_{\varphi}-\frac{\pi}{2}}\} \right)\\  \nonumber \hspace{5mm}
&+Q^{(n)}\left(
\left\{z\in \mathbb{R}^2\setminus\{x\}\left|\right. z_1-x_1\geq 0,z_2-x_2 \geq 0,  (z-x)^{\top}\varphi = 0\right\}\right) \\ \nonumber
=&Q^{(n)}\left(x+A_{t_{\varphi}+\frac{\pi}{2}}\right)-Q^{(n)}\left(x+A_{t_{\varphi}-\frac{\pi}{2}}\right) \\  \nonumber\hspace{5mm}
&+Q^{(n)}\left(
\left\{z\in \mathbb{R}^2\left|\right. z_1-x_1\geq 0,z_2-x_2\geq 0,  (z-x)^{\top}\varphi = 0\right\}\right).
\end{align} 
The third term converges uniformly in $\varphi$ due to Lemma \ref{lemma:lines}. Analyzing the  convergence of the first two summands reduces to studying 
\begin{align*}
F_{Q^{(n)}_{\Pi_x}}(t)
=Q^{(n)}\left(\Pi_x^{-1}\left[0, t\right]\right)
=Q^{(n)}(x+A_t)
\end{align*}
for $t\in(0,2 \pi]$.
The left-hand side  of the above formula corresponds to a distribution function converging uniformly to the continuous limit function $F_{Q_{\Pi_x}}(t)$. 

Now, let $\varphi_1>0$ (case 2).
Then, we have 
\begin{align*}
a_n(\varphi)=&Q^{(n)}\left(\left\{z\in \mathbb{R}^d|(z-x)^T\varphi\geq 0\right\}\right)=1-Q^{(n)}\left(\left\{z\in \mathbb{R}^d|(z-x)^T\varphi< 0\right\}\right)\\
=&1- Q^{(n)}\left(\left\{z\in \mathbb{R}^d|(z-x)^T(-\varphi)> 0\right\}\right)\\
=&1-Q^{(n)}\left(\left\{z\in \mathbb{R}^d|(z-x)^T(-\varphi)\geq 0\right\}\right)\\
&+Q^{(n)}\left(\left\{z\in \mathbb{R}^d|(z-x)^T(-\varphi)= 0\right\}\right)\\
=&1-a_n(-\varphi)+Q^{(n)}\left(\left\{z\in \mathbb{R}^d|(z-x)^T(-\varphi)= 0\right\}\right).
\end{align*}
The last summand on the right-hand side of the above equation converges uniformly in $\varphi$ due to Lemma \ref{lemma:lines}.

Finally, let $\varphi_1=0$ (case 3).
In this case, we have only two points on the sphere and hence uniform convergence on these points follows immediately. 
Putting all results together, we obtain: 
\begin{align*}
\sup_{\varphi \in S^2}\left|a_n(\varphi)-a(\varphi)\right|
\longrightarrow 0
\end{align*}
and consequently
\begin{align*}
D_{Q_{Y_n}}(x)=\inf_{\varphi\in S^2}a_n(\varphi)\longrightarrow \inf_{\varphi\in S^2}a(\varphi)=D_{Q_{Y}}(x).
\end{align*}
\end{proof}

Finally, we present the lemmas which are used in order to prove Proposition \ref{prop:conv_of_depth_point}.
Lines, respectively half lines, play an important role. The following lemma allows to handle them in the context of uniform convergence. 
\begin{lemma}\label{lemma:lines}
Let $Q^{(n)}$, $Q$ and $\Pi_x$ be as above.  $\varphi_1\leq 0$. Then
\begin{align*}
&Q^{(n)}\left(\left\{z\in \mathbb{R}^2\left|\right.  (z-x)^{\top}\varphi = 0, z_2-x_2 > 0\right\}\right)  \\
&\longrightarrow Q\left(\left\{z\in \mathbb{R}^2\left|\right.  (z-x)^{\top}\varphi = 0, z_2-x_2 > 0\right\}\right) 
\end{align*}
and
\begin{align*}
Q^{(n)}\left(\left\{z\in \mathbb{R}^2\left|\right.  (z-x)^{\top}\varphi = 0 \right\}\right)  \longrightarrow Q\left(\left\{z\in \mathbb{R}^2\left|\right.  (z-x)^{\top}\varphi = 0\right\}\right) 
\end{align*}
uniformly in $\varphi$. 
\end{lemma}
\begin{proof}
Knowing that the distribution function $F_{Q^{(n)}_{\Pi_x}}(t)$ converges pointwise towards a continuous limit, we conclude that it converges uniformly. 
From this fact,  we can also derive that the jumps
\[
F_{Q^{(n)}_{\Pi_x}}(t)-F_{Q^{(n)}_{\Pi_x}}(t-) =Q^{(n)}_{\Pi_x} (\{t\})
\]
converge uniformly in $t$.  Hence,
\[
Q^{(n)}\left(\left\{z\in \mathbb{R}^2\left|\right.  (z-x)^{\top}\varphi = 0, z_2-x_2 > 0\right\}\right) = Q^{(n)}_{\Pi_x} (\{t_{\varphi}-\frac{\pi}{2}\})
\]
converges uniformly in $\varphi$.  Analogously, this holds for $z_2-x_2 < 0$.  For $z_2-x_2 = 0$ we only obtain a single point, and hence, no dependence on $\varphi$. This yields uniform convergence of the sum of the three functions
\[
\varphi \mapsto Q^{(n)}\left(\left\{z\in \mathbb{R}^2\left|\right.  (z-x)^{\top}\varphi = 0 \right\}\right) 
\]
which is the measure of the full line.
\end{proof}

\begin{lemma}\label{lemma:parametrization}
Let $\varphi=(\varphi_1, \varphi_2)\in S^2$ with
$\varphi_1<0$.
For
$t_{\varphi}:=
\arctan\left(\frac{\varphi_2}{\varphi_1}\right)+\pi$
 it holds that
\begin{align*}
&\left\{z\in \mathbb{R}^2\left|\right. z^{\top}\varphi\geq 0\right\}\\
=&A_{t_{\varphi}+\frac{\pi}{2}}\setminus A_{t_{\varphi}-\frac{\pi}{2}}\cup \left\{z\in \mathbb{R}^2\left|\right. z_1\geq 0,z_2\geq 0,  z^{\top}\varphi= 0\right\},
\end{align*}
where for $t\in [0, 2\pi)$ the subset $A_t$ 
 of $\mathbb{R}^2$ is defined by 
\begin{align*}
A_{t}:=W^{-1}\left(\pi_2^{-1}\left[0, t\right]\right).
\end{align*} 
\end{lemma}
\begin{proof}
Consider $\varphi=(\varphi_1, \varphi_2)\in S^2$ with
$\varphi_1<0$ and $\varphi_2>0$.
Then, for $t_{\varphi}=\pi_2\circ W(\varphi)=\arctan\left(\frac{\varphi_2}{\varphi_1}\right)+\pi$,
it follows that
\begin{align*}
A_{t_{\varphi}+\frac{\pi}{2}}
=&W^{-1}\left(\pi_2^{-1}\left[0, t_{\varphi}+\frac{\pi}{2}\right]\right)
=W^{-1}\left((0, \infty)\times\left[0, t_{\varphi}+\frac{\pi}{2}\right]\right)
\\
=&\left\{z\in \mathbb{R}^2\setminus\{(0, 0)\}\left|\right. W(z)\in (0, \infty)\times\left[0, t_{\varphi}+\frac{\pi}{2}\right]\right\}\\
=&\left\{z\in \mathbb{R}^2\setminus\{(0, 0)\}\left|\right. W(z)\in (0, \infty)\times\left[0, \frac{\pi}{2}\right)\right\}\\
&\cup 
\left\{z\in \mathbb{R}^2\setminus\{(0, 0)\}\left|\right. W(z)\in (0, \infty)\times \frac{\pi}{2}\right\}\\
&\cup 
\left\{z\in \mathbb{R}^2\setminus\{(0, 0)\}\left|\right. W(z)\in (0, \infty)\times\left(\frac{\pi}{2}, t_{\varphi}+\frac{\pi}{2}\right]\right\}\\
=&\left\{z\in \mathbb{R}^2\setminus\{(0, 0)\}\left|\right.  z_1>0, z_2\geq 0\right\}\\
&\cup 
\left\{z\in \mathbb{R}^2\setminus\{(0, 0)\}\left|\right. z_1=0, z_2> 0\right\}\\
&\cup 
\left\{z\in \mathbb{R}^2\setminus\{(0, 0)\}\left|\right. z_1<0, \arctan\left(\frac{z_2}{z_1}\right)+\pi\leq \arctan\left(\frac{\varphi_2}{\varphi_1}\right)+\frac{3\pi}{2}\right\}.
\end{align*}
Since $\tan\left(x+\frac{\pi}{2}\right)=-\frac{1}{\tan(x)}$, the third set can be written as follows: 
\begin{align*}
&\left\{z\in \mathbb{R}^2\setminus\{(0, 0)\}\left|\right. z_1<0, \arctan\left(\frac{z_2}{z_1}\right)+\pi\leq \arctan\left(\frac{\varphi_2}{\varphi_1}\right)+\frac{3\pi}{2}\right\}\\
=&\left\{z\in \mathbb{R}^2\setminus\{(0, 0)\}\left|\right. z_1<0, \arctan\left(\frac{z_2}{z_1}\right)\leq \arctan\left(\frac{\varphi_2}{\varphi_1}\right)+\frac{\pi}{2}\right\}\\
=&\left\{z\in \mathbb{R}^2\setminus\{(0, 0)\}\left|\right. z_1<0, \frac{z_2}{z_1}\leq -\frac{\varphi_1}{\varphi_2} \right\}\\
=&\left\{z\in \mathbb{R}^2\setminus\{(0, 0)\}\left|\right. z_1<0, z_1 \varphi_1+z_2\varphi_2\geq 0\right\}\\
=&\left\{z\in \mathbb{R}^2\setminus\{(0, 0)\}\left|\right. z_1<0, z^{\top}\varphi\geq 0\right\}.
\end{align*}
Moreover, we have
\begin{align*}
A_{t_{\varphi}-\frac{\pi}{2}}
=&W^{-1}\left(\pi_2^{-1}\left[0, t_{\varphi}-\frac{\pi}{2}\right]\right)
=W^{-1}\left((0, \infty)\times\left[0, t_{\varphi}-\frac{\pi}{2}\right]\right)
\\
=&\left\{z\in \mathbb{R}^2\setminus\{(0, 0)\}\left|\right. W(z)\in (0, \infty)\times\left[0, t_{\varphi}-\frac{\pi}{2}\right]\right\}\\
=&\left\{z\in \mathbb{R}^2\setminus\{(0, 0)\}\left|\right.  z_1 >0, z_2\geq 0, \arctan\left(\frac{z_2}{z_1}\right)\leq t_{\varphi}-\frac{\pi}{2}\right\}\\
=&\left\{z\in \mathbb{R}^2\setminus\{(0, 0)\}\left|\right.  z_1 >0, z_2\geq 0, \arctan\left(\frac{z_2}{z_1}\right)\leq \arctan\left(\frac{\varphi_2}{\varphi_1}\right)+\frac{\pi}{2}\right\}\\
=&\left\{z\in \mathbb{R}^2\setminus\{(0, 0)\}\left|\right.  z_1 >0, z_2\geq 0, \frac{z_2}{z_1}\leq -\frac{\varphi_1}{\varphi_2}\right\}\\
=&\left\{z\in \mathbb{R}^2\setminus\{(0, 0)\}\left|\right. z_1>0,z_2\geq 0,  z_1 \varphi_1+z_2\varphi_2\leq 0\right\}\\
=&\left\{z\in \mathbb{R}^2\setminus\{(0, 0)\}\left|\right. z_1>0,z_2\geq 0,  z^{\top}\varphi\leq 0\right\}.
\end{align*}
As a result, it holds that
\begin{align*}
A_{t_{\varphi}+\frac{\pi}{2}}\setminus A_{t_{\varphi}-\frac{\pi}{2}}
=&\left\{z\in \mathbb{R}^2\setminus\{(0, 0)\}\left|\right.  z_1>0, z_2\geq 0, z^{\top}\varphi> 0\right\}\\
&\cup 
\left\{z\in \mathbb{R}^2\setminus\{(0, 0)\}\left|\right. z_1=0, z_2> 0\right\}\\
&\cup
\left\{z\in \mathbb{R}^2\setminus\{(0, 0)\}\left|\right. z_1<0, z^{\top}\varphi\geq 0\right\}.
\end{align*}
Therefore, we have
\begin{align*}
&\left\{z\in \mathbb{R}^2\left|\right. (z-x)^{\top}\varphi\geq 0\right\}\\
=&\{A_{t_{\varphi}+\frac{\pi}{2}}+x\}\setminus \{A_{t_{\varphi}-\frac{\pi}{2}}+x\}\cup \left\{z\in \mathbb{R}^2\left|\right. z_1-x_1\geq 0,z_2-x_2\geq 0,  (z-x)^{\top}\varphi= 0\right\}.
\end{align*}
Consider $\varphi=(\varphi_1, \varphi_2)\in S^2$ with
$\varphi_1<0$ and $\varphi_2 < 0$. Then, for
 $t_{\varphi}=\pi_2\circ W(\varphi)=\arctan\left(\frac{\varphi_2}{\varphi_1}\right)+\pi$ it follows that 
\begin{align*}
A_{t_{\varphi}+\frac{\pi}{2}}
=&W^{-1}\left(\pi_2^{-1}\left[0, t_{\varphi}+\frac{\pi}{2}\right]\right)
=W^{-1}\left((0, \infty)\times\left[0, t_{\varphi}+\frac{\pi}{2}\right]\right)
\\
=&\left\{z\in \mathbb{R}^2\setminus\{(0, 0)\}\left|\right. W(z)\in (0, \infty)\times\left[0, t_{\varphi}+\frac{\pi}{2}\right]\right\}\\
=&\left\{z\in \mathbb{R}^2\setminus\{(0, 0)\}\left|\right. W(z)\in (0, \infty)\times\left[0, \frac{\pi}{2}\right)\right\}\\
&\cup 
\left\{z\in \mathbb{R}^2\setminus\{(0, 0)\}\left|\right. W(z)\in (0, \infty)\times \frac{\pi}{2}\right\}\\
&\cup 
\left\{z\in \mathbb{R}^2\setminus\{(0, 0)\}\left|\right. W(z)\in (0, \infty)\times\left(\frac{\pi}{2}, \frac{3\pi}{2}\right)\right\}\\
&\cup 
\left\{z\in \mathbb{R}^2\setminus\{(0, 0)\}\left|\right. W(z)\in (0, \infty)\times \frac{3\pi}{2}\right\}\\
&\cup 
\left\{z\in \mathbb{R}^2\setminus\{(0, 0)\}\left|\right. W(z)\in (0, \infty)\times\left(\frac{3\pi}{2}, t_{\varphi}+\frac{\pi}{2}\right]\right\}\\
=&\left\{z\in \mathbb{R}^2\setminus\{(0, 0)\}\left|\right.  z_1>0, z_2\geq 0\right\}\\
&\cup 
\left\{z\in \mathbb{R}^2\setminus\{(0, 0)\}\left|\right. z_1=0, z_2> 0\right\}\\
&\cup 
\left\{z\in \mathbb{R}^2\setminus\{(0, 0)\}\left|\right. z_1<0\right\}\\
&\cup
\left\{z\in \mathbb{R}^2\setminus\{(0, 0)\}\left|\right. z_1=0,  z_2<0\right\}\\
&\cup
\left\{z\in \mathbb{R}^2\setminus\{(0, 0)\}\left|\right. z_1>0,  z_2<0, \arctan\left(\frac{z_2}{z_1}\right)+2\pi\leq t_{\varphi}+\frac{\pi}{2}\right\}.
\end{align*}
Since $\tan\left(x-\frac{\pi}{2}\right)=-\frac{1}{\tan(x)}$, the fifth set can be written as follows: 
\begin{align*}
&\left\{z\in \mathbb{R}^2\setminus\{(0, 0)\}\left|\right. z_1>0,  z_2<0, \arctan\left(\frac{z_2}{z_1}\right)+2\pi\leq t_{\varphi}+\frac{\pi}{2}\right\}\\
&=
\left\{z\in \mathbb{R}^2\setminus\{(0, 0)\}\left|\right. z_1>0,  z_2<0, \arctan\left(\frac{z_2}{z_1}\right)\leq \arctan\left(\frac{\varphi_2}{\varphi_1}\right)-\frac{\pi}{2}\right\}\\
&=\left\{z\in \mathbb{R}^2\setminus\{(0, 0)\}\left|\right. z_1>0,  z_2<0, \frac{z_2}{z_1}\leq -\frac{\varphi_1}{\varphi_2} \right\}\\
&=\left\{z\in \mathbb{R}^2\setminus\{(0, 0)\}\left|\right. z_1>0,  z_2<0, z_2\varphi_2+z_1\varphi_1\geq 0 \right\}\\
&=\left\{z\in \mathbb{R}^2\setminus\{(0, 0)\}\left|\right. z_1>0,  z_2<0, z^{\top}\varphi\geq 0 \right\}.
\end{align*}
Moreover, we have
\begin{align*}
A_{t_{\varphi}-\frac{\pi}{2}}
=&W^{-1}\left(\pi_2^{-1}\left[0, t_{\varphi}-\frac{\pi}{2}\right]\right)\\
=&
W^{-1}\left((0, \infty)\times\left[0, \frac{\pi}{2}\right)\right)\\
&\cup
W^{-1}\left((0, \infty)\times \frac{\pi}{2}\right)\\
&\cup
W^{-1}\left((0, \infty)\times\left(\frac{\pi}{2}, t_{\varphi}-\frac{\pi}{2}\right]\right)
\\
=&\left\{z\in \mathbb{R}^2\setminus\{(0, 0)\}\left|\right.  z_1>0, z_2\geq 0\right\}\\
&\cup 
\left\{z\in \mathbb{R}^2\setminus\{(0, 0)\}\left|\right. z_1=0, z_2> 0\right\}\\
&\cup 
\left\{z\in \mathbb{R}^2\setminus\{(0, 0)\}\left|\right. z_1<0, \arctan\left(\frac{z_2}{z_1}\right)+\pi\leq \arctan\left(\frac{\varphi_2}{\varphi_1}\right)+\frac{\pi}{2}\right\}.
\end{align*}
Since $\tan\left(x-\frac{\pi}{2}\right)=-\frac{1}{\tan(x)}$, the third set can be written as follows: 
\begin{align*}
&\left\{z\in \mathbb{R}^2\setminus\{(0, 0)\}\left|\right. z_1<0, \arctan\left(\frac{z_2}{z_1}\right)+\pi\leq \arctan\left(\frac{\varphi_2}{\varphi_1}\right)+\frac{\pi}{2}\right\}\\
&=
\left\{z\in \mathbb{R}^2\setminus\{(0, 0)\}\left|\right. z_1<0, \arctan\left(\frac{z_2}{z_1}\right)\leq \arctan\left(\frac{\varphi_2}{\varphi_1}\right)-\frac{\pi}{2}\right\}\\
&=\left\{z\in \mathbb{R}^2\setminus\{(0, 0)\}\left|\right. z_1<0, \frac{z_2}{z_1}\leq -\frac{\varphi_1}{\varphi_2}\right\}\\
&=\left\{z\in \mathbb{R}^2\setminus\{(0, 0)\}\left|\right. z_1<0, z^{\top}\varphi\leq 0\right\}.
\end{align*}
So far, we covered the cases $\varphi_1<0$, $\varphi_2>0$ and $\varphi_1<0$, $\varphi_2<0$.
For $\varphi_1<0$ and $\varphi_2=0$, we have
$t_{\varphi}=\pi$.
It follows that
\begin{align*}
A_{t_{\varphi}+\frac{\pi}{2}}
=&W^{-1}\left(\pi_2^{-1}\left[0, \frac{3\pi}{2}\right]\right)
=W^{-1}\left((0, \infty)\times\left[0, \frac{3\pi}{2}\right]\right)
\\
=&\left\{z\in \mathbb{R}^2\setminus\{(0, 0)\}\left|\right.  z_1>0, z_2\geq 0\right\}\\
&\cup 
\left\{z\in \mathbb{R}^2\setminus\{(0, 0)\}\left|\right. z_1=0, z_2> 0\right\}\\
&\cup 
\left\{z\in \mathbb{R}^2\setminus\{(0, 0)\}\left|\right. z_1<0\right\}\\
&\cup
\left\{z\in \mathbb{R}^2\setminus\{(0, 0)\}\left|\right. z_1=0,  z_2<0\right\}.
\end{align*}
Moreover, we have
\begin{align*}
A_{t_{\varphi}-\frac{\pi}{2}}
=&W^{-1}\left(\pi_2^{-1}\left[0, \frac{\pi}{2}\right]\right)\\
=&\left\{z\in \mathbb{R}^2\setminus\{(0, 0)\}\left|\right.  z_1>0, z_2\geq 0\right\}\\
&\cup 
\left\{z\in \mathbb{R}^2\setminus\{(0, 0)\}\left|\right. z_1=0, z_2> 0\right\}.
\end{align*}
As a result, it holds that
\begin{align*}
A_{t_{\varphi}+\frac{\pi}{2}}\setminus A_{t_{\varphi}-\frac{\pi}{2}}
=&\left\{z\in \mathbb{R}^2\setminus\{(0, 0)\}\left|\right. z_1<0\right\}\\
&\cup
\left\{z\in \mathbb{R}^2\setminus\{(0, 0)\}\left|\right. z_1=0,  z_2<0\right\}.
\end{align*}
Therefore, we have
\begin{align*}
&\left\{z\in \mathbb{R}^2\left|\right. (z-x)^{\top}\varphi\geq 0\right\}\\
=&\left\{z\in \mathbb{R}^2\left|\right. z_1-x_1\leq 0\right\}\\
=&\{A_{t_{\varphi}+\frac{\pi}{2}}+x\}\setminus \{A_{t_{\varphi}-\frac{\pi}{2}}+x\}\\
&\cup \left\{z\in \mathbb{R}^2\left|\right. z_1-x_1\geq 0,z_2-x_2\geq 0,  (z-x)^\top\varphi= 0\right\}.
\end{align*}
\end{proof}

\bibliographystyle{plainnat}
\bibliography{literature}

\end{document}